\documentclass[12pt]{amsart}

\usepackage{graphicx}
\usepackage{comment}
\usepackage{latexsym,amsmath,amsthm,amssymb,amsfonts,MnSymbol,epsfig}
\usepackage[initials]{amsrefs}
\usepackage[all]{xy}

\usepackage{xcolor}

\newtheorem{theorem}{Theorem}

\newtheorem{proposition}[theorem]{Proposition}

\theoremstyle{definition}
\newtheorem{definition}[theorem]{Definition}

\theoremstyle{remark}
\newtheorem{remark}[theorem]{Remark}
\newtheorem{example}[theorem]{Example}

\numberwithin{equation}{section}

\newcommand{\ca}{\mathcal{A}}
\newcommand{\cb}{\mathcal{B}}

\newcommand{\cl}{\mathcal{L}}

\newcommand{\ck}{\mathcal{K}}
\newcommand{\lp}{\mathcal{L}^p}

\newcommand{\g}{\gamma}
\newcommand{\G}{\Gamma}

\newcommand{\ci}{\sqrt{2 i}}

\DeclareMathOperator{\ind}{index}

\DeclareMathOperator{\End}{End}

\DeclareMathOperator{\Ch}{Ch}
\DeclareMathOperator{\tr}{tr}
\DeclareMathOperator{\ord}{ord}
\DeclareMathOperator{\Tr}{Tr}
\DeclareMathOperator{\PF}{PF}

\DeclareMathOperator{\Ker}{Ker}

\DeclareMathOperator{\Dom}{Dom}
\DeclareMathOperator{\Res}{Res}
\DeclareMathOperator{\vol}{vol}

\providecommand{\CC}{{\mathbb{C}}}
\providecommand{\RR}{{\mathbb{R}}}
\providecommand{\QQ}{{\mathbb{Q}}}

\providecommand{\ZZ}{{\mathbb{Z}}}

\newcommand{\lra}{\longrightarrow}

\providecommand{\BB}{{\mathcal B}} % bounded operators
\providecommand{\KK}{{\mathcal K}} % compact operators
\providecommand{\LL}{{\mathcal L}} % adjointable operators
\providecommand{\TTX}{{\mathbb T}X} % tangent groupoid
\newcommand{\del}{{\partial}}
\newcommand{\delbar}{{\bar{\partial}}}

\newcommand{\Ind}{{\mathrm{Index}}}
\newcommand{\coker}{{\mathrm{coker}}}
\newcommand{\ch}{{\mathrm{ch}}}

\providecommand{\Td}{{\mathrm{Td}}}
\providecommand{\ev}{{\mathrm{ev}}}
\author{Alexander Gorokhovsky}
\address{University of Colorado, Boulder, Campus Box 395, Boulder, Colorado, 80309, USA}
\email{gorokhov@euclid.colorado.edu}
\author{Erik van Erp}
\address{Dartmouth College, 6188, Kemeny Hall, Hanover, New Hampshire, 03755, USA}
\email{jhamvanerp@gmail.com}

\title{Index theory and Noncommutative Geometry: A Survey }

\begin{document}
\maketitle

 \section{Introduction}

The contributions of Alain Connes to index theory are fundamental, broad, and deep.
This article is an introductory  survey of index theory in the context of noncommutative geometry.
%and it aims to highlight  key  aspects of that work.
The selection of topics is determined by the interests of the authors, and we make no claim to being exhaustive.

This survey  has  two parts.
In the first half  we consider index theory in the setting of $K$-theory.
If $P$ is an elliptic linear differential operator on a closed smooth manifold $X$,
then the highest order part of $P$ determines a cohomological object that is most naturally understood as an element in  $K$-theory,
\[ [\sigma(P)]\in K^0(T^*X)\]
The index of $P$ is a function of this $K$-theory class $[\sigma(P)]$.
Atiyah and Singer derived a topological formula for the index map
 \[ K^0(T^*X)\to \ZZ \qquad [\sigma(P)]\mapsto \Ind\, P\]
Developments in noncommutative geometry have shed considerable light on the nature of this map.
Here we emphasize, on the one hand, the organizing role of Connes' tangent groupoid and its generalizations,
and on the other hand his study of the wrong-way functoriality of the Gysin map in $K$-theory.
%An indispensable tool in this approach to index theory is Gennadi Kasparov's $KK$-theory, a far-reaching generalization of an idea of Atiyah, who in \cite{At69} axiomatized the notion of an abstract ``elliptic'' operator on a compact Hausdorff space, to serve as an explicit model of a cycle in $K$-homology.
An important fruit of these investigations  of index theory in noncommutative geometry is the Baum-Connes conjecture, which concerns the $K$-theory of the reduced $C^*$-algebra of locally compact groups.
This work has deep ties to geometry and algebraic topology.

A  different and purely analytic  approach to  index theory is based on the McKean-Singer formula \cite{McKean-Singer} which expresses the index of an elliptic operator $P$ in terms of the trace of the heat kernel,
\[ \Ind\,P = \Tr e^{-tP^*P}-\Tr e^{-tPP^*}\qquad t>0 \]
The heat kernel has an asymptotic expansion in powers of $t$ (as $t\to 0^+$).
This implies that its trace, and hence the index of $P$, can be computed as the integral of a density which is locally determined by the coefficients of the operator $P$.
This density is called the {\em local index} of $P$.

For Dirac-type operators the local index
density can be explicitly computed, and is identified with differential forms that represent the characteristic classes that appear in the Atiyah-Singer formula.
Thus, one obtains not just a formula for the index, which is a {\em global} invariant, but a representation of the index as the integral of a  density that is canonically and {\em locally} determined by the operator.
This makes it possible to extend index theory to manifolds with boundary, as in the Atiyah-Patodi-Singer  theorem \cite{APS}.
Bismut extended this approach to elliptic families \cite{Bismut}, using the superconnection approach developed for this purpose by Quillen \cite{Quillen}. The local index formula in this case describes an explicit differential form representing the   Chern character of an elliptic family.

The extension of local index theory to the noncommutative framework necessitates the construction  of the Chern character for  noncommutative algebras. In this case, a natural receptacle for the Chern character is cyclic homology theory. Cyclic theory has been closely connected to  index theory from its very beginning \cite{cn85}.

The second half of the present paper focuses on a more recent development: the local index formula of Connes and Moscovici in the context of noncommutative geometry.
Our goal is to give a brief yet detailed exposition of the proof of this theorem. We follow closely the arguments of the original paper \cite{CM95}.
A distinctive feature of the Connes-Moscovici paper is its detailed and illuminating explicit calculations, illustrating the cocycle property of the local index formula, and the renormalization procedure used to derive it.
For  an alternative approach to the Connes-Moscovici formula we refer the reader to \cite{Higson1, Higson2}.

\section{Atiyah and Singer}

In 1963, Michael Atiyah and Isadore Singer announced their formula for the index of an elliptic  operator on a compact manifold \cite{AS63}.
On a  manifold $X$ with  complex vector bundles $E, F$,
a linear differential operator of order $d$
\[ P:C^\infty(X,E)\to C^\infty(X,F) \]
is, in local coordinates on $U\subset X$ and with local trivializations of $E, F$, of the form
%\[ P = \sum_{|\alpha|\le d} a_\alpha(x)\del^\alpha : C^\infty(U,\CC^r)\to C^\infty(U,\CC^r)\]
%where $\alpha = (\alpha_1, \dots, \alpha_n)$, $|\alpha| = \alpha_1+\cdots+\alpha_n$, $\del^\alpha = \del_1^{\alpha_1}\cdots \del_n^{\alpha_n}$, $\del_k=\del/\del x_k$,
%\[ P = \sum_{\alpha_1+\cdots +\alpha_n\le d} a_\alpha(x)\del_1^{\alpha_1}\cdots \del_n^{\alpha_n} : C^\infty(U,\CC^r)\to C^\infty(U,\CC^r)\]
\[ P = \sum_{\alpha_1+\cdots +\alpha_n=0}^d\; a_\alpha\;\frac{\del^{\alpha_1}}{\del x_1^{\alpha_1}}\,\cdots \,\frac{\del^{\alpha_n}}{\del x_n^{\alpha_n}} : C^\infty(U,\CC^r)\to C^\infty(U,\CC^r)\]
where the coefficients $a_\alpha$ are  smooth functions on $U$ with values in the algebra $M_r(\CC)$ of $r\times r$ matrices.
The principal symbol of $P$ is a matrix-valued polynomial in the formal variable $\xi=(\xi_1, \dots, \xi_n)\in \RR^n$,
\[ \sigma(P)(x,\xi) = \sum_{\alpha_1+\cdots +\alpha_n= d} a_\alpha(x)\,\xi_1^{\alpha_1}\,\cdots\, \xi_n^{\alpha_n}\; \in M_k(\CC)\]
By definition, the operator $P$ is elliptic if the matrix $\sigma(P)(x,\xi)\in M_k(\CC)$ is invertible for all $\xi\ne 0$, and at all points of $X$.
%Examples of elliptic operators are Laplacians, the Dolbeault operator, and the Dirac operator.
If $X$ is a compact manifold without boundary, an elliptic operator on $X$ has finite dimensional kernel and cokernel.
The index of $P$ is
\[ \Ind P = \dim\ker P - \dim\coker P= \dim\ker P - \dim\ker P^*\]
Interpreting $(x,\xi)$ as  coordinates for the cotangent bundle $T^*X$,
$\sigma(P)$ has an invariant interpretation as an endomorphism of vector bundles
\[ \sigma(P) : \pi^*E\to \pi^*F \qquad \pi:T^*X\to X\]
where $\pi^*E, \pi^*F$ are $E, F$ pulled back to $T^*X$.
Since $\sigma(P)$ is invertible outside the zero section of $T^*X$, which is a compact subset, the triple $(\pi^*E, \pi^*F,  \sigma(P))$ determines an element in the Atiyah-Hirzebruch $K$-theory group of $T^*X$,
\[ [\sigma(P)] \in K^0(T^*X)\]
The Chern character
\[ \ch: K^j(T^*X)\otimes \QQ \to \bigoplus_{k=0}^n H_c^{2k+j}(T^*X, \QQ) \qquad j=0,1\]
is a natural ring isomorphism
of (rational) $K$-theory and cohomology with compact supports.
%Assuming $X$ is oriented, we have the Thom isomorphism (integration in the fibers),
%\[ \phi : H_c^{n+k}(T^*X, \QQ) \cong H^k(X,\QQ)\]
%Thus, the principal symbol of $P$ determines a rational cohomology class on $X$,
%\[ \phi(\ch(\sigma(P))) \in H^\bullet(X, \QQ)\]
%The topological formula for the index of $P$ announced in \cite{AS63}  is
%\[ \Ind P =  \{\phi(\ch(\sigma(P)))\cup \Td(TX\otimes \CC)\}[X]\]
%\begin{equation}\label{eqn:AS}
% \Ind P =  \{\phi(\ch(\sigma(P)))\cup \Td(TX\otimes \CC)\}[X]
%\end{equation}
%where $\Td$ is the Todd class,  $TX\otimes \CC$ the complexified tangent bundle of $X$,
%and $[X]\in H_n(X,\ZZ)$ the fundamental cycle for the chosen orientation of $X$.
%The same formula is valid if $P$ is an elliptic (classical) pseudodifferential operator.
%If $X$ is not oriented,  the equivalent formula
\begin{proposition}\cite{AS63}\label{eqn:AS}
If $P$ is an elliptic operator on a closed manifold $X$ then
\[ \Ind\, P =  \{\ch(\sigma(P))\cup \Td(\pi^*TX\otimes \CC)\}[T^*X]\]
where $\Td$ is the Todd class,  $TX\otimes \CC$ the complexified tangent bundle of $X$,
and $[T^*X]\in H_{2n}(T^*X,\QQ)$ the fundamental cycle for $T^*X$, canonically oriented as a symplectic manifold.
\end{proposition}

The formula of Atiyah and Singer answered a question of Gel\cprime fand, who proposed in \cite{Ge60} the investigation of the relation between homotopy invariants and analytical invariants of elliptic operators.
Moreover, a number of classical theorems in geometry and topology are special cases of the Atiyah-Singer index formula.
If we take
\[ P=d+d^* : C^\infty(X,\Lambda^{\rm even} T^*M) \to C^\infty(X,\Lambda^{\rm odd} T^*M)\]
with $d$ the deRham differential, we recover the curvature formula of Gauss-Bonnet for the Euler number of a closed oriented surface, and its generalization to higher dimensional manifolds due to Hopf and Chern.
With
\[P=\delbar+\delbar^* :  C^\infty(X,\Lambda^{\rm even} T^{1,0}M) \to C^\infty(X,\Lambda^{\rm odd} T^{1,0}M)\]
with $\delbar$ the Dolbeault operator, we get  the Riemann-Roch theorem for Riemann surfaces and Hirzebruch's generalization  to projective algebraic varieties.
%The Atiyah-Singer formula implies that the Hirzebruch-Riemann-Roch theorem  holds even for almost  complex manifolds.
Finally, if $P$ is the signature operator we obtain Hirzebruch's signature theorem.

\section{The Gysin map}

The original proof of Proposition \ref{eqn:AS} followed the strategy of Hirzebruch's proof of the signature theorem,
 relying on the calculation of  cobordism groups by Ren\'e Thom.
The proof published in \cite{AS1} is independent of cobordism theory, and makes use of Atiyah-Hirzebruch $K$-theory.

If $X$ is compact Hausdorff, $K^0(X)$ is the abelian group generated by isomorphism classes of complex vector bundles on $X$.
Vector bundles naturally pull back along continuous maps.
$K$-theory is a cohomology theory for locally compact Hausdorff spaces,
and a contravariant functor for proper maps $f:X\to Y$.
If $X, Y$ are smooth manifolds, then
$K$-theory is also {\em covariant} for (not necessarily proper)  smooth maps $f:X\to Y$ that are $K$-oriented.
A $K$-orientation of a smooth map $f:X\to Y$ is a spin$^c$ structure for the vector bundle $TX\oplus f^*TY$.
%For an embedding $f:X\to Y$,
%a $K$-orientation of $f$ is equivalent to a spin$^c$ structure for the normal bundle of $f(X)$ in $Y$.
%A $K$-orientation for a submersion $f:X\to Y$ is equivalent to a spin$^c$ structure for the kernel of  $df:TX\to TY$. In other words, the fibers of a $K$-oriented submersion are spin$^c$ manifolds.
A $K$-orientation of $f$  induces a wrong-way functorial Gysin map
\[ f_! :K^\bullet(X)\to K^\bullet(Y)\]
In three special cases, the construction of the Gysin map is straightforward.
If $f:U\to V$ is the inclusion of $U$ as an open submanifold of $V$, then  $f_!$ is the map in analytic $K$-theory determined by the inclusion  $C_0(U)\subset C_0(V)$.
If $E$ is the total space of a spin$^c$ vector bundle on $Z$,
and $f:Z\to E$ is the zero section,
then $f_!$ is the $K$-theory Thom isomorphism.
Similarly, if $f:E\to Z$ is the base point projection of a spin$^c$ vector bundle, then $f_!$ is the inverse of the Thom isomorphism.
In general, every $K$-oriented map can be factored as a composition of three maps of the above type.

\vskip 6pt
The cotangent bundle  $T^*X$ of a closed smooth manifold $X$ has a canonical symplectic form, and is therefore a spin$^c$ manifold.
Thus, the map $p:T^*X\to {\rm point}$ is $K$-oriented, and we have a Gysin map
\[ p_!:K^0(T^*X) \to K^0({\rm point}) =\ZZ\]
%In \cite{AS1} Atiyah and Singer prove that if $P$ is an elliptic operator on a compact manifold $X$, then
\begin{proposition}\label{prop:AS1}\cite{AS1}
If $P$ is an elliptic operator on a closed manifold $X$ then
\[ \Ind\,P = p_!(\sigma(P))\]
%where $p!$ is the Gysin map determined by $p:T^*X\to {\rm point}$, and $p$ is $K$-oriented by the canonical symplectic structure of $T^*X$.
\end{proposition}

A purely topological calculation, carried out  in \cite{AS3},  shows that the characteristic class formula of Proposition \ref{eqn:AS} follows from Proposition \ref{prop:AS1}.
%as the Gysin map (or ``shriek-map'') determined by an embedding $i:X\hra \RR^{n}$,
%\[ i_!: K^0(TX)\to K^0(T\RR^n)\cong \ZZ\]
%Here $K^0(T\RR^n)\cong K^0(\RR^{2n}) \cong \ZZ$ is the Bott isomorphism.
%Bott periodicity plays a fundamental role in all proofs of the index theorem based on $K$-theory.

\section{The tangent groupoid}\label{sec:tangent}

In \cite[II.5]{Co94} Connes gives a short and conceptually elegant proof that the map $(\ev_1)_\ast\circ (\ev_0)_\ast^{-1}:K^0(T^*X)\to \ZZ$ determined by the tangent groupoid agrees with the Gysin map $p_!:K^0(T^*X)\to \ZZ$.
This proves Proposition \ref{prop:AS1}.
Connes' proof makes use of the Baum-Connes isomorphism for the tangent groupoid $G=\TTX$,
\[ K^0(BG)\cong K_0(C^*(G))\]
(see section \ref{sec:BC}). Here $BG$ is a classifying space for free and proper actions of $G$.
This isomorphism replaces the groupoid $G=\TTX$ by the topological space $BG$,
which  turns   the index map into one that can be computed topologically.

In subsequent work by many authors, the tangent groupoid has proven  to be a very useful tool in index theory.
Analogs of the tangent groupoid have been constructed to deal with index problems in various contexts.
As a random (neither exhaustive nor fully representative)  sample of publications ranging from 1987 to 2018, see \cite{HS87} \cite{MP97} \cite{NWX99} \cite{ANS06} \cite{CR09} \cite{vE10a} \cite{Hi10} \cite{DS18}.

\begin{remark}
The tangent groupoid formalism is intimately connected to  deformation theory.
The $C^*$-algebra $C^*(\TTX)$ is the algebra of sections of a continuous field of $C^*$-algebras over $[0,1]$ with fiber $C_0(T^*X)$ at $t=0$ and $\KK(L^2(X))$ at all $t>0$.
This field encodes a strong deformation quantization of the algebra of classical observables (functions on the symplectic phase space $T^*X$)  to quantum observables (operators on $L^2(X)$).
\end{remark}

\section{$K$-homology}

The wrong-way functoriality of the Gysin map  suggests that the topological index is perhaps more naturally regarded as a map in $K$-homology.
Atiyah-Hirzebruch $K$-theory is associated to the Bott spectrum,
\[ K^j(X) = [X, K_j]\qquad K_{2j} = BU\times \ZZ, \; K_{2j+1}=U\]
Abstractly,  $K$-homology $K_j(X)$ is then defined  as the stable homotopy group
of  the smash product of $X$ with the Bott spectrum,
\[ K_j^{top}(X) = \pi_j(X\wedge \underline{K})=\lim_{\lra} \pi_{j+k} (X\wedge K_k)\]
$K$-homology is naturally covariant.
The Gysin map $p_!:K^0(T^*X)\to \ZZ$ in $K$-theory corresponds in $K$-homology to the map determined by $\epsilon:X\to {\rm point}$,
\[ \epsilon_* : K_0^{top}(X)\to K_0^{top}({\rm point})\cong \ZZ\]
composed with the  Poincar\'e duality isomorphism
\[ K^0(T^*X)\cong K_0^{top}(X)\]

\subsection{Analytic $K$-homology}

The homotopy theoretic definition of $K$-homology lacks a concrete model for cycles in $K_j(X)$ that is useful in calculations.
In \cite{At69} Atiyah shows that  elliptic operators can represent $K$-cycles.
If $P$ is an elliptic operator on a closed manifold $X$,
and $V$ is a complex vector bundle on $X$,
we can twist $P$ by $V$ to obtain a new elliptic operator
\[ P_V : C^\infty(M,E\otimes V)\to C^\infty(M,F\otimes V)\]
%This is the cap product
%\[ K_0(X)\times K^0(X) \to K_0(X) \qquad (P, V)\mapsto P_V\]
Thus, $P$ determines a homomorphism
\[ K^0(X) \to \ZZ\qquad [V]\mapsto \Ind \,P_V\]
In \cite{At69} Atiyah  axiomatized the notion of elliptic operator such that it still makes sense  if $X$ is not  a smooth manifold, but  only a  compact Hausdorff space.
%and shows that elements in $K_0(X)$ can be represented by such abstract elliptic operators.
Following Atiyah, an element in the even $K$-homology group $K_0(X)$ is represented by a bounded Fredholm operator  $F:H^0\to H^1$,
where the Hilbert spaces $H^j$ are equipped with a $\ast$-representation $\phi_j:C(X)\to \BB(H^j)$,
and $F$ is {\em `an operator on X'} in the sense that
\begin{equation}\label{eqn:Atiyah}
F\phi_0(f) - \phi_1(f)F: H^0\to H^1 \;\text{is a compact operator for all} \; f\in C(X)
\end{equation}
This axiom guarantees that the index of the operator $F$ twisted by a vector bundle $V$ on $X$ is a well-defined integer.
In fact, twisting implements the cap product, i.e.
the structure of $K$-homology as a module over the $K$-theory ring,
\[ \cap : K^0(X)\times K_0(X)\to K_0(X) \qquad [V] \cap [P]= [P_V]\]
For an elliptic operator $P:C^\infty(E)\to C^\infty(F)$, the Hilbert spaces $H^0, H^1$ are appropriate Sobolev spaces of sections in $E, F$, such that $P:H^0\to H^1$ is a bounded Fredholm operator,
and Atiyah's axiom is satisfied because of the Rellich Lemma.
The correct equivalence relation for Atiyah's group of abstract elliptic operators was identified by Gennadi Kasparov \cite{Ka75}.
The  Atiyah-Kasparov  analytic $K$-homology of $X$ is the $KK$-group $KK(C(X),\CC)$.
For finite $CW$-complexes $X$ there is an isomorphism
\[ KK(C(X),\CC)\cong K_0^{top}(X)\]
Kasparov developed the theory far beyond Atiyah's  ansatz.
For two separable (possibly noncommutative) $C^*$-algebras $A, B$, elements in the abelian group $KK(A,B)$ are represented by generalized Fredholm operators  $F:H^0\to H^1$.
Now $H^0, H^1$ are Hilbert $C^*$-modules over $B$, and $\phi_j:A\to \LL(H^j)$ are $\ast$-representations of $A$ by adjointable operators on $H^j$.
$F$ is invertible modulo ``compact operators'' in the sense of the theory of operators on Hilbert-modules.
The  power of $KK$-theory rests on the existence of an associative product
\[ \# : KK(A,B)\times KK(B,C)\to KK(A,C)\]
For example, $K$-theory is
\[ K_0(A)\cong KK(\CC,A)\]
and an element in the bivariant group $\xi \in KK(A,B)$
induces, by right composition, a map of $K$-theory groups
\[ K_0(A)\to K_0(B) : \alpha\mapsto \alpha\# \xi\]
Kasparov showed that  Proposition  \ref{eqn:AS} of Atiyah-Singer can be understood
in $K$-homology as follows.
The Dirac operator $D_{T^*X}$ of the (symplectic) spin$^c$ manifold $T^*X$ determines a class in analytic $K$-homology,
\[ [D_{T^*X}]\ \in KK(C_0(T^*X), \CC)\]
and hence determines a map $K^0(T^*X)\to \ZZ$.
This map is equal to the  topological index of Atiyah-Singer.
\footnote{This is a restatement of Bott's version of the Atiyah-Singer formula in his review of \cite{AS63} for the AMS mathematical reviews.}
But the  principal symbol $\sigma(P)$ of an elliptic operator $P$ determines not only an element in $K$-theory,
\[ [\sigma(P)]\in KK(\CC, C_0(T^*X))\]
but in fact gives the slightly better
\[ [\sigma(P)]\in KK(C(X), C_0(T^*X))\]

\begin{proposition} \cite[section 24.5]{Bl98}\label{prop:KK}
The element in analytic $K$-homology determined by an elliptic operator $P$ on a closed manifold $X$
%\[ [P] \in KK(C(X),\CC)\]
is the product
\[ [P] = [\sigma(P)]\# [D_{T^*X}]\in KK(C(X),\CC)\]
\end{proposition}
The importance of $KK$-theory to index theory is hard to overstate.
For an excellent  first introduction to $KK$-theory see \cite{Hi90}.
%The literature is too extensive to include references here.

\subsection{Geometric $K$-homology}
A geometric model for cycles in $K$-homology was developed by Paul Baum and Ron Douglas \cite{BD80}.
In their theory, a $K$-cycle is represented by a triple $(M,V,\varphi)$,
where $M$ is a closed spin$^c$ manifold,
$V$ is a complex vector bundle on $M$,
and $\varphi:M\to X$ is a continuous map.
The non-trivial aspect of the theory is the equivalence relation on such triples (see \cite{BD80} for details).
The groups $K_0^{\rm geo}(X)$ and $K_1^{\rm geo}(X)$ consist of equivalence classes of triples $(M,V,\varphi)$ with $M$ even or odd dimensional, respectively.
Remarkably, Baum-Douglas geometric $K$-homology is equivalent to topological $K$-homology for {\em all} topological spaces $X$ \cite{Ja98},
\[ K_j^{\rm geo}(X) \cong \pi_j(X\wedge \underline{K})\]
A proof of the special case of the Atiyah-Singer index theorem for Dirac operators using geometric $K$-homology is as follows.
Proposition \ref{prop:KK} and \ref{prop:Dirac} together imply Proposition \ref{eqn:AS}.

\begin{proposition}\label{prop:Dirac}
If $M$ is a spin$^c$ manifold with Dirac operator $D$, and $V$ is a smooth complex vector bundle on $M$, then
\[ \Ind\, D_V = \{\ch(V)\cup \Td(M)\}[M]\]
where $\Td(M)$ is the Todd class of the spin$^c$ vector bundle $TM$.
\end{proposition}
\begin{proof}
The geometric $K$-homology of a point $K^{\rm geo}_0({\rm point})$
consists of pairs $(M,V)$, with $M$ a spin$^c$ manifold with complex vector bundle $V$.
There is an isomorphism of abelian groups,
\[ K^{\rm geo}_0({\rm point})\cong \ZZ\qquad (M,V) \mapsto \{\ch(V)\cup \Td(M)\}[M]\]
The proof that this is an isomorphism is purely topological, and Bott periodicity plays a key role.
There is a second homomorphism, namely the analytic index
\[ K^{\rm geo}_0({\rm point})\to \ZZ\qquad (M,V) \mapsto  \Ind\, D_V\]
The proof that this map is well-defined relies  on bordism invariance of the index of Dirac operators.
%\footnote{In fact, the index of the spin$^c$ Dirac operator is a genus.}
To prove that the two homomorphisms of abelian groups $K^{\rm geo}_0({\rm point})\to \ZZ$ are equal,
it suffices to verify this by direct calculation in one example where $D_V$ has non-zero index (for example  the Dolbeault operator of $\CC P^1$).
(See \cite{BvE18} for details.)

\end{proof}

The point of view proposed in  \cite{BD80} is that index theory, in general, is based on the equivalence between analytic and geometric $K$-homology.
For a compact manifold $X$ (or, more generally, a finite $CW$-complex), there is an isomorphism \cite{BHS07}
\[ \mu:K_0^{\rm geo}(X)\cong  KK(C(X),\CC)\]
If $(M,V,\varphi)$ is a geometric $K$-cycle for $X$, the spin$^c$ manifold $M$ has a Dirac operator $D$;
$D$ twisted by $V$ determines an element in analytic $K$-homology $[D_V]\in KK(C(M),\CC)$;
finally $\phi:M\to X$ maps this element to
\[ \mu(M,V,\varphi) = \varphi_*(D_V) \in KK(C(X),\CC)\]
In \cite{BD80}, Baum and Douglas conceptualize index theory as follows:
\begin{quote}
\it Given an element in analytic $K$-homology, $\xi\in KK_j(C(X),\CC)$, explicitly compute the unique $\tilde{\xi}$ in geometric $K$-homology, $\tilde{\xi}\in K_j^{\rm geo}(X)$, corresponding to $\xi$.
\end{quote}
In concrete terms, this  proposes that  index problems are solved by reduction to the index of a Dirac operator, possibly on a different manifold---as, for example, in Proposition \ref{prop:KK}.
This perspective on index theory was used, more recently,  to solve the index problem for hypoelliptic  operators in the Heisenberg calculus on contact manifolds \cite{BvE14}.
This is a class of (pseudo)differential operators that are Fredholm, but not elliptic.
Yet these operators satisfy Atiyah's axiom (\ref{eqn:Atiyah}), and therefore determine an analytic $K$-cycle.
The index problem for this class of hypoelliptic operators was solved by computing an equivalent geometric $K$-cycle.

\section{Elliptic families}

The suggestion in \cite{BD80}  that index problems are solved by reduction to Dirac operators is intimately related to  the point of view that the topological index is a Gysin map in $K$-theory.
%At bottom, what is at play here is the duality between the Dirac operator of $\RR^n$ as a $K$-homology element and the Bott generator in $K$-theory,
%\[ [D_{\RR^n}]\in KK_n(C_0(\RR^n),\CC)\cong \ZZ\quad \beta_n\in  KK_n(\CC,C_0(\RR^n))\cong \ZZ\]
To see  the  connection between the two perspectives, we consider the families index theorem of Atiyah and Singer \cite{AS4}.

Suppose $P_b$ is a smooth family of elliptic operators on a closed manifold $X$, parametrized by points $b\in B$ in another compact manifold $B$.
Alternatively, we may think of such an elliptic family $\{P_b,b\in B\}$ as a single differential operator $P$ on $X\times B$ that differentiates  along  the fibers of $X\times B\to B$.
The kernel $V_b=\ker\,P_b$ of each individual operator $P_b$ is a finite dimensional vector space.
If the dimension of $V_b$ is independent of $b\in B$,
then the collection $\{V_b,b\in B\}$ defines a vector bundle $V=\ker\, P$ on $B$.
The same holds for the cokernels of $P_b$, and the index of the elliptic family $P$ is a formal difference of two vector bundles on $B$,
\[ \Ind \,P =[\ker\,P]-[\coker\,P]\in K^0(B)\]
The construction can be modified so that it makes sense even if the dimension of $V_b$ is not constant.
%Further, the same construction works if $P$ is an elliptic family on a locally trivial fiber bundle $\pi:Z\to B$.
See \cite{AS4}, where the product $X\times B$ is generalized to fiber bundles $Z\to B$, and $B$ is allowed to be a compact Hausdorff space.

The analog of Proposition \ref{prop:KK} holds for families.
As a single operator, $P:H^0\to H^1$ is a generalized Fredholm operator, where $H^0, H^1$ are  Hilbert modules over the $C^*$-algebra $C(B)$.
As such, $P$ determines a class in $KK(\CC,C(B))$.
The families index of $P$ is precisely the isomorphism
\[ \Ind: KK(\CC, C(B))\cong K^0(B)\]
If $V\to Z$ is a complex vector bundle on $Z$,
then each elliptic operator $P_b$ can be twisted by the restriction $V|Z_b$.
We  obtain the twisted family $P_V$, whose families index is an element in $K^0(B)$.
Thus, $P$ determines a group homomorphism
\[ [V]\to \Ind P_V:K^0(Z)\to K^0(B)\]
This homomorphism is the product, in $KK$-theory, with a $KK$-element determined by the elliptic family $P$,
\[  \Ind P_V = [V]\# [P]\qquad [P]\in KK(C(Z),C(B))\]
As with a single operator, the principal symbol of an elliptic family on $Z\to B$ is an element in $K$-theory,
\[ [\sigma(P)] \in K^0(T_{\rm vert}^*Z)\]
where $T_{\rm vert}Z$ is the bundle of vertical tangent vectors, i.e. the kernel of the map $d\pi:TZ\to TB$.
%For a product $Z=X\times B$ simply $T_{\rm vert}^*Z=T^*X\times B$.
We have in fact
\[ [\sigma(P)] \in KK(C(Z), C_0(T_{\rm vert}^*Z))\]
The cotangent bundles $T^*Z_b$ of the fibers $Z_b$ are symplectic and hence spin$^c$ manifolds.
The family of Dirac operators $D_{T^*Z_b}$ is an elliptic family parametrized by $b\in B$,
and determines  a $KK$-element
\[ [D_{T^*_{\rm vert}Z}] \in KK(C_0(T^*_{\rm vert} Z), C(B))\]
Then, as in Kasparov's Proposition \ref{prop:KK}, a version of the families index theorem is
\[ [P] =  [\sigma(P)]\# [D_{T^*_{\rm vert}Z}] \in KK(C(Z), C(B))\]
%Just as (\ref{eqn:AS}) follows if we combine Propositions \ref{prop:KK} and \ref{prop:Dirac},
To derive the families index theorem of \cite{AS4}, we need a topological characterization of the $KK$-element $[D_{T^*_{\rm vert}Z}]$.
This topological characterization is provided by the Gysin map.

\begin{proposition}\label{prop:Gysin}
Let  $\pi:W\to B$ be a submersion of manifolds.
Assume that $\ker d\pi\subset TW$ is a spin$^c$ vector bundle, so that each fiber $W_b$ is a spin$^c$ manifold, and $\pi$ is $K$-oriented.
If   $D$ is the elliptic family of Dirac operators $D_b$ of  $W_b$,
then the $KK$-element
\[ [D]\in KK(C_0(W),C(B))\]
has the property that the corresponding map in $K$-theory
\[ K^0(W)\to K^0(B): [V]\mapsto [V] \# [D]=\Ind\, D_V\]
is  the Gysin map
\[  \pi_!: K^0(W)\to K^0(B)\]
\end{proposition}
Proposition \ref{prop:Gysin} is  a special case of the families index theorem in \cite{AS4}.
Together with the equality $[P] =  [\sigma(P)]\# [D_{T^*_{\rm vert}Z}]$,
it gives the  index theorem for general elliptic families as,
\[ \Ind P = \pi_!(\sigma(P))\in K^0(B)\qquad \pi:T^*_{\rm vert}Z\to B\]

\section{The adiabatic groupoid}

The tangent groupoid formalism can be adapted to the families index problem.
The relevant   generalization of Connes' tangent groupoid is developed by Nistor, Weinstein, and Xu in  \cite{NWX99}.
For an an arbitrary Lie groupoid $G$,
let $T^sG$ be the vector bundle on $G$ of vectors tangent to the source fibers, i.e. vectors in the kernel of $ds:TG\to TG^{(0)}$, where $s:G\to G^{(0)}$ is the source map of the groupoid $G$.
Let  $AG$ be  the restriction of $T^sG$ to the space of objects $G^{(0)}\subset G$.
The vector bundle $AG$ has the structure of a Lie algebroid, but this is not relevant for our purposes here.

A ``blow-up'' construction that closely follows \cite[II.5]{Co94} produces a groupoid
\[ G_{\rm ad} = AG \times \{0\} \,\sqcup\, G\times (0,1]\]
called the {\em adiabatic groupoid} of $G$.
Algebraically, $G_{ad}$ is the disjoint union of a family of copies of $G$ parametrized by $t\in (0,1]$
with the vector bundle $AG$ as the boundary at $t=0$.
%The ``blow-up'' of the unit space $G^{(0)}$ in $G$ is a direct generalization of the construction in \cite[II.5]{Co94}.
Connes' tangent groupoid is the special case with $G=X\times X$ and $AG=TX$.

As in section  \ref{sec:tangent} above,
the adiabatic groupoid gives rise to a map in $K$-theory
\[ K^0(A^*G)\to K_0(C^*(G))\]
from the $K$-theory of the vector bundle $A^*G$ to the $K$-theory of the Lie groupoid $G$.
This map can be interpreted as an analytic index.
If $P$ is a (pseudo)differential operator on $G$ that differentiates only in the direction of the fibers of the source map $s:G\to G^{(0)}$, and if $P$ is right invariant,
then if the principal symbol $\sigma(P)(x,\xi)$ is invertible  for all groupoid units $x\in G^{(0)}$ and all $\xi\in A^*G$ with $\xi\ne 0$, then  $P$ has an index,
\[ \Ind \,P \in K_0(C^*(G))\]
The principal symbol of $P$  determines a class in topological $K$-theory
\[ [\sigma(P)]\in K^0(A^*G)\]
The analytic index map determined by the adiabatic groupoid maps $[\sigma(P)]$ to $\Ind \,P$.
For differential operators, the proof of these facts is essentially the same as the proof of Proposition \ref{prop:g-ind}.

\begin{example}
Given a smooth fiber bundle $\pi:Z\to B$,
consider the Lie groupoid
\[ G=Z\times_B Z = \{ (z,z')\in Z\times Z\mid \pi(z)=\pi(z')\}\]
with multiplication $(z,z')(z',z'')=(z,z'')$.
The Lie algebroid of $G$ is $AG=T_{\rm vert}Z$.
The groupoid $Z\times_B Z$ is Morita equivalent to the manifold $B$ (as a Lie groupoid in which all elements are units), and the adiabatic groupoid of $G$ determines an index map
\[ K^0(A^*G) = K^0(T^*_{\rm vert}Z)\to K^0(B) \cong K_0(C^*(G))\]
This map is the analytic index map for families of elliptic operators on the fibers of $Z\to B$.
\end{example}

\section{Foliations}

A foliation of a smooth manifold $Z$ is a subbundle $F\subseteq TZ$ such that $[F,F]\subseteq F$.
Equivalently, $Z$ is a disjoint union of immersed submanifolds $L\subset Z$, called the leaves of the foliation,
such that at every point $x\in L$ for any leaf $L\subset Z$ we have $T_xL = F_x$.
The leaf space   is the set of leaves with the quotient topology inherited from $Z$.
If  $\pi:Z\to B$ is a fibration, then $Z$ is foliated with $F=T_{\rm vert}Z=\ker d\pi$.
The leafs are the fibers $Z_b=\pi^{-1}(b)$, and the leaf space is $B$.
In general, the leaf space of a foliation may have pathological topology.
For example, if $Z$ contains at least one leaf $L$ that is dense in $Z$ (e.g. the Kronecker foliation of a torus), then the leaf space has only two open sets.

The holonomy groupoid $G_{Z/F}$ of a foliated manifold $(Z,F)$ is a Lie groupoid
whose elements are equivalence classes of paths $\gamma:[0,1]\to L$ connecting  points $x=s(\gamma), y=r(\gamma)$ in the same leaf $L$.
Two paths are equivalent if they have the same holonomy (see \cite{Co80}).
Multiplication in $G_{Z/F}$ is composition of paths.
In simple cases where there is no holonomy (e.g. the Kronecker foliation),
the groupoid $G_{Z/F}$ is algebraically a disjoint union of the pair groupiods $L\times L$ of the leaves.
If the  leaf space (with its quotient topology) is Hausdorff, then it is Morita equivalent to the groupoid $G_{Z/F}$.
But in general  the holonomy groupoid $G_{Z/F}$ (as a ``smooth stack'')  is a better representation of the space of leaves than the leaf space with its quotient space.

By definition, the $K$-theory of the leaf space of a foliation is the $K$-theory of the reduced $C^*$-algebra of the holonomy groupoid,
\[ K^0(Z/F):= K_0(C_r^*(G_{Z/F}))\]
The Lie algebroid of $G_{Z/F}$ is $AG=F$, as a vector bundle over $G^{(0)}=Z$.
Thus, by the general procedure of  \cite{NWX99}, we have an analytic index
\[ K^0(F^*)\to K^0(Z/F)\]
This map generalizes the analytic index of  elliptic families.
It is defined, for example, for differential operators on $Z$ that differentiate in the leaf direction,
and are elliptic along the leafs.
Such operators are called {\em longitudinally elliptic}.

%In \cite{Co80}, Connes generalized Proposition \ref{prop:Gysin} by constructing an element in $KK$-theory that implements the Gysin map for an arbitrary $K$-oriented smooth map $f:X\to Y$,
%\[ f_!\in KK(C(X),C(Y))\]
%\[ f_!:K^j(X)\to K^j(Y)\qquad f_!(\xi) = \xi\# f_!\]
Assume that the leaves
are even dimensional spin$^c$ manifolds.
This means that $F$ is a spin$^c$ vector bundle of even rank, and so in particular we have the Thom isomorphism,
\[ K^0(F^*)\cong K^0(Z)\]
The spin$^c$ Dirac operators of the leaves form a longitudinally elliptic family $D$ on $Z$.
Given a vector bundle $V$ on $Z$, we can twist the Dirac operator of each leaf $L$
by the restriction  $V|L$, and we obtain a new elliptic family $D_V$.
The longitudinal index problem consists in identifying a ``topological index''
that equals the analytic index map
\[ K^0(Z)\to K^0(Z/F)\qquad V\mapsto \Ind D_V\]
The solution was suggested in  \cite{Co80} and proven in  \cite{CS84} by Connes and Skandalis.
A  smooth maps between foliated manifolds
\[ f:Z_1/F_1\to Z_2/F_2\]
is defined as a smooth groupoid homomorphism $G'_{Z_1/F_1}\to G_{Z_2/F_2}$,
where  $G'_{Z_1/F_1}$ is a groupoid that is Morita equivalent to $G_{Z_1/F_1}$.
(This is a morphism of stacks.)
The notion of $K$-orientation naturally extends to such maps.
The difficulty is showing that such a $K$-oriented map determines a Gysin map in $K$-theory
that is (wrong-way) functorial,
\[ f_!:K^\bullet(V_1/F_1)\to K^\bullet(V_2/F_2)\]
If $F$ is a spin$^c$ vector bundle, then the identity map $l:Z\to Z$ is $K$-oriented when considered as a morphism from the manifold $Z$ to the foliation $Z/F$.
In the case of a fiber bundle $\pi:Z\to B$, where $Z/F=B$,
this is just the map $\pi$.
The topological index for longitudinally elliptic operators is the Gysin map
\[ l_!:K^0(Z)\to K^0(Z/F)\]
Thus, the index theorem of Connes-Skandalis for longitudinally elliptic operators generalizes Proposition  \ref{prop:Gysin}.

\section{The Baum-Connes conjecture}\label{sec:BC}

Among the most significant developments of index theory in the context of noncommutative geometry is the Baum-Connes conjecture.
To fit it in the framework discussed so far,
consider  a manifold  $X$ with fundamental cover $p:\tilde{X}\to X$.
The fundamental groupoid of $X$ is the Lie groupoid
\[ \tilde{X}\times_X \tilde{X} = \{(a,b)\in \tilde{X}\times \tilde{X}\mid p(a)=p(b)\}\]
with multiplication $(a,b)(b,c)=(a,c)$.
Equivalently, elements of the fundamental groupoid are homotopy classes of paths $\gamma:[0,1]\to X$ (with no base point),
and  multiplication is composition of paths.
The fundamental groupoid is Morita equivalent to the fundamental group $\Gamma=\pi_1(X)$, and so
\[ K_0(C^*(\tilde{X}\times_X \tilde{X})) \cong K_0(C^*(\Gamma))\]
The Lie algebroid of $\tilde{X}\times_X \tilde{X}$ is  $TX$, i.e. it is the same as the Lie algebroid of the pair groupoid $X\times X$.
The adiabatic groupoid of $\tilde{X}\times_X \tilde{X}$ thus gives rise to an index map in $K$-theory
\[ K^0(T^*X) \to K_0(C^*(\Gamma))\]
If $P$ is an elliptic operator on $X$,
then $P$ lifts to a $\Gamma$-equivariant elliptic operator $\tilde{P}$ on $\tilde{X}$.
The analytic index maps the symbol $[\sigma(P)]\in K^0(T^*X)$
to the $\Gamma$-index of $\tilde{P}$.
By Poincar\'e duality, we may  conceive of this $\Gamma$-index as a map
\[  K_0(X) \to K_0(C^*(\Gamma))\]
(where $K_0(X)$ is $K$-homology, $K_0(C^*(\Gamma))$ is $K$-theory).

Now let $\Gamma$ be an arbitrary countable discrete group.
The classifying space $B\Gamma$ for principal $\Gamma$-bundles is not generally a smooth manifold,
but one can still define a generalized  index map
\[ \mu :K_0(B\Gamma) \to K_0(C^*(\Gamma))\]
called the assembly map.
By composition with the natural map $C^*(\Gamma)\to C^*_r(\Gamma)$,  we may replace the full $C^*$-algebra $C^*(\Gamma)$ on the right-hand side by the reduced $C^*$-algebra $C^*_r(\Gamma)$,
\[ \mu_r :K_0(B\Gamma) \to K_0(C^*_r(\Gamma))\]
The conjecture asserts that $\mu_r$ is an isomorphism if $\Gamma$ is a countable discrete group without torsion elements \cite{Ka84}.
If $\Gamma$ has torsion, the left hand side of the conjecture  is  replaced
by the $\Gamma$-equivariant $K$-homology  (with $\Gamma$-compact supports) of the classifying space $\underline{E}\Gamma$ for {\em proper} (instead of principal) $\Gamma$-actions.
If $\Gamma$ is torsion free then $K_0^\Gamma(\underline{E}\Gamma)\cong K_0(B\Gamma)$.
In general, the assembly map
\[ \mu:K_j^\Gamma(\underline{E}\Gamma)\to K_j(C^*_r(\Gamma))\qquad j=0,1\]
is conjectured to be an isomorphism for {\em any} second countable locally compact  group $\Gamma$ \cite{BCH94}.
Early versions of the Baum-Connes conjecture concerned the $K$-theory of groupoids (the holonomy groupoid of a foliation  \cite{Co80} and crossed product groupoids \cite{BC88}),
but we now have counterexamples for the conjecture for groupoids.
For groups, to this day no counterexample has been found, and the conjecture has been verified for large classes of groups.
Since injectivity of the assembly map for a discrete group $\Gamma$ implies the Novikov higher signature conjecture,
the Baum-Connes conjecture is a development within noncommutative geometry
with significant implications in algebraic topology.

%llokokl,,,,,,,,,,,,,,,,,,,,,'
%ooooooooooooooooooooooo   ff nnnnbbb

\section{Cohomological index formulas and local index theory}
One can obtain index formulas for the index of a single operator or an elliptic family by applying the Chern character to $K$-theoretical index formulas.
A different analytic approach has been proposed by Atiyah-Bott \cite{Atiyh-Bott}. It is based on the following fundamental observation. Let $P$, as before, be an elliptic operator on a manifold $X$ acting between sections of bundles $E$ and $F$. We can then form two positive operators $P^*P$ and $P P^*$ and  consider  two $\zeta$-functions $\Tr (1+P^*P)^{-s}$ and $\Tr (1+P^*P)^{-s}$. Both of these functions are  are defined and holomorphic for $\Re s$ sufficiently large and admit meromorphic continuation to the entire complex plane. Moreover, they are holomorphic at $s=0$ and the value of each $\zeta$-function at $s=0$ can be computed by an explicit integral of a well defined local density,  see \cite{seely}. The operators $P^*P$ and $P P^*$ have the same sets of nonzero eigenvalues with the same multiplicities. The multiplicities of $0$ as an eigenvalue on the other hand differ for two operators, and the difference of multiplicities is precisely the
index of operator $P$. It follows that for any $s$ with $\Re s$ sufficiently large
\[
\Tr (1+P^*P)^{-s} - \Tr (1+P P^*)^{-s} = \ind P
\]
Writing $\zeta(s):= \Tr (1+P^*P)^{-s} - \Tr (1+P P^*)^{-s}$ we thus obtain by analytic continuation
\[\ind P = \zeta (0).\]
It is useful to rewrite these considerations in the following $\mathbb{Z}_2$ -graded notations. Let
\begin{equation}\label{Hilb}
 \mathcal{H}: = L^2(X, E)\oplus L^2(X, F)
\end{equation} be a Hilbert space. On $\mathcal{H}$ one has a grading operator $\gamma$ given by the matrix $\begin{bmatrix}
                                                           1 & 0 \\
                                                           0 & -1
                                                         \end{bmatrix}$ with respect to the decomposition  \eqref{Hilb} and an operator $D = \begin{bmatrix}
                                                                                                                                               0 & P^* \\
                                                                                                                                               P & 0
                                                                                                                                             \end{bmatrix}$.
In these terms we can write $\zeta(s) = \Tr \gamma(1+D^2)^{-s}$.
McKean and Singer \cite{McKean-Singer} proposed using a a closely related formula
\[
\ind P= \Tr \gamma e^{-tD^2},
\]
which holds for any $t>0$, for index computation. Since the trace of the heat kernel $e^{-tD^2}$ admits an asymptotic expansion as $t\to 0^+$ with the coefficients being canonically defined
local densities $X$ this formula allows in principle to compute the index as an integral over $X$.
Explicitly, the index formula for Dirac-type operators has been obtained by computing  the local densities using invariant theory in \cites{Gilkey, Atiyah-Bott-Patodi}. Later  significantly simpler methods allowing direct calculations of the local index density were found   in \cites{Getzler1, Getzler2, Berlin-Vergne}, cf. \cite{bgv}.
It is important to note that this approach provides not just cohomology classes appearing in a cohomological index formula, but rather canonically defined differential forms representing these classes, thus justifying the name local index formulas. This locality has been exploited to extend index theory e.g. to manifolds with boundary \cite{APS}.
The extension of local index techniques to the elliptic families case has been carried out in \cite{Bismut} based on the superconnection formalism proposed in \cite{Quillen}. Further extension of these results to the case of foliations presents an additional difficulty. Since the algebras involved are noncommutative, de Rham theory is no longer a proper receptacle for the Chern character in $K$-theory. As we will see, this role is played by cyclic theory.
Finally, we mention that a different approach to local index theory based on the study of cyclic theory of deformation quantization algebras has been developed by R. Nest and B. Tsygan. For their approach as well as some applications we refer to the original papers \cites{nt1, nt2, nt3, bnt}.

\section{Cyclic complexes}	
In this section we give a very brief overview of periodic cyclic homological and cohomological complexes, mostly to fix the notations. The standard reference for this material is \cite{loday}.

For a complex unital algebra $\ca$ set $C_l(\ca):= \ca \otimes(\ca/(\mathbb{C} \cdot 1))^{\otimes l}$, $l \ge 0$.
For a topological algebra, e.g. a normed algebra, one needs to take an appropriately completed tensor product. One defines differentials $b\colon C_l(\ca) \to C_{l-1}(\ca)$ and $B \colon C_l(\ca) \to C_{l+1}(\ca)$ by
\[
b (a_0\otimes a_1\otimes \ldots a_l):= \sum_{i=0}^{l-1} (-1)^ia_0\otimes \ldots a_ia_{i+1}\otimes \ldots a_l+(-1)^la_la_0\otimes a_1\otimes \ldots a_{l-1}
\]
\[
B (a_0\otimes a_1\otimes \ldots a_l):= \sum_{i=0}^{l} (-1)^{li}1 \otimes a_i\otimes a_{i+1}\otimes \ldots a_{i-1}\text{ (with $a_{-1}:=a_l)$}.
\]
One verifies directly that $b$, $B$ are well defined and satisfy $b^2=0$, $B^2=0$, $Bb+bB=0$. We will be primarily interested in the periodic cyclic complex, which is a totalization of a bicomplex defined as follows. Set $C_{kl}(\ca):= C_{l-k}(\ca)$, $k, l \in \mathbb{Z}$ ($C_{kl}(\ca)=0$ if $k>l$) and note that $b$, $B$ define maps $b \colon C_{kl}(\ca) \to C_{k, (l-1)}(\ca) $ and  $ B\colon C_{kl}(\ca) \to C_{(k-1), l}(\ca) $. $ \left(C_{\bullet, \bullet}(\ca), b, B\right)$ is thus a bicomplex.
To obtain the periodic cyclic complex it is totalized as follows:
\[CC^{per}_i(\ca):= \prod_{k+l=i} C_{kl}(\ca)
\]
with the differential $b+B \colon CC^{per}_i(\ca) \to CC^{per}_{i-1}(\ca)$. Note $CC^{per}_i(\ca)=  CC^{per}_{i+2m}(\ca)$ for every $m \in \mathbb{Z}$, and therefore the complex is (indeed) $2$-periodic. Its homology is denoted by $HC^{per}_0(\ca) (\cong HC^{per}_{2m}(\ca))$ and $HC^{per}_1(\ca) (\cong HC^{per}_{1+2m}(\ca))$, $m \in \mathbb{Z}$.
We will write a chain in $CC^{per}_i(\ca)$, $i=0$ or $1$ as
$ \alpha = \sum_{m=0}^{\infty} \alpha_{i+2m}, \text{ where }  \alpha_k \in C_k(\ca)$.

Two important examples of classes in cyclic homology are the following. Let $e \in M_n(\ca)=\ca \otimes M_n(\mathbb{C})$ be a idempotent. Then
\[
\Ch(e):=\tr e+\sum_{l=1}^{\infty} (-1)^l \frac{(2l)!}{l!} \tr \left(e-\frac{1}{2}\right)\otimes e^{\otimes 2l} \in \prod_{l\ge 0} C_{-l, l}(\ca)=CC^{per}_0(\ca)
\]
where $\tr \colon (\ca \otimes M_n(\mathbb{C}))^{\otimes k} \to \ca^{\otimes k}$ is the map given by
\[
\tr (a_0\otimes m_0)\otimes (a_1\otimes m_1) \otimes \ldots (a_k\otimes m_k)= (\tr m_0m_1\ldots m_k) a_0  \otimes a_1 \otimes \ldots a_k.
\]
The class of $\Ch (e)$ in $HC^{per}_0(\ca)$ depends only on the class of $e$ in $K_0(\ca)$ and thus defines the Chern character morphism
\[
\Ch \colon K_0(\ca) \to HC^{per}_0(\ca)
\]
Similarly, if $u \in M_n(\ca)$ is invertible on can define
\[
\Ch(u):= \frac{1}{\sqrt{2 \pi i}} \sum_{l=0}^{\infty}(-1)^l l! \tr (u^{-1}\otimes u)^{\otimes (l+1)} \in CC^{per}_1(\ca)
\]
This defines  a homomorphism
\[
\Ch \colon K_1(\ca) \to HC^{per}_1(\ca).
\]
from  topological $K$-theory to  periodic cyclic cohomology in the odd case.
Note that one can define Chern characters for $K_0$ and $K_1$ by the same formulas in the case of algebraic $K$-theory as well, but in the case of $K_1$ the target of the Chern character should be a different flavor of cyclic theory, the negative cyclic homology.
\begin{example}
Let $\ca= C^\infty(X)$ where $X$ is a compact manifold. We have a map $\lambda \colon C_k(\ca) \to \Omega^k(X)$ given by
\[ \lambda (a_0\otimes a_1 \otimes \ldots a_k) = \frac{1}{k!} a_0da_1\ldots da_k
\]
which intertwines the differential $B$ with de Rham differential $d$ and $b$ with $0$. It follows that if we consider a bicomplex $\mathcal{D}_{kl}: = \Omega^{l-k}(X)$ with the differentials given by $d$ and $0$, the map $\lambda$ will induce a morphism of bicomplexes $C_{kl}(\ca)$ and $\mathcal{D}_{kl}$. It can be shown that this map is a quasi-isomorphism \cite{cn85} and hence $HC^{per}_i(C^\infty(X)) \cong \oplus_{m \in \mathbb{Z}} H^{i+2m} (X)$. The map $K^i(X) \cong K_i(C^\infty(X))  \xrightarrow{\Ch}HC^{per}_i (C^\infty(X)) \xrightarrow{\lambda}\oplus_{m \in \mathbb{Z}} H^{i+2m} (X)$ recovers the ordinary Chern character in $K$-theory (up to normalization).

\end{example}
Dually, one considers cyclic cohomology. Set $C^k(\ca) := (C_k(\ca))'$ -- the space of continuous multilinear functionals $\phi(a_0, a_1,\ldots, a_k)$ on $\ca$ with the property that if for some  $i \ge 1$  $a_i=1$, then $\phi(a_0, a_1,\ldots, a_k)=0$. The transpose of the differentials $b$, $B$ induce maps, also denoted $b$, $B$: %define maps
$b \colon C^k(\ca) \to C^{k+1}(\ca)$, $B \colon C^k(\ca) \to C^{k-1}(\ca)$. One then forms a bicomplex $C^{kl}(\ca):= C^{l-k}(\ca)$, $k, l \in \mathbb{Z}$ with the differentials $b \colon C^{kl}(\ca) \to C^{k, (l+1)}(\ca) $ and  $ B\colon C^{kl}(\ca) \to C^{(k+1), l}(\ca) $. Again dually to the homological case it is totalized using the direct sums rather than products:
\[CC_{per}^i(\ca):= \bigoplus_{k+l=i} C^{kl}(\ca).
\]
Here again the complex we consider is $2$-periodic. An even (resp. odd) cyclic cochain is thus given by a collection of multilinear functionals
$\phi_k=\phi_k (a_0,\ldots, a_k) \in C^k(\ca)$, $k=0,2 \ldots$ (resp. $k=1,3,\ldots$), only finitely many of which are nonzero. A cochain is a cocycle if it satisfies
$b\phi_k+B\phi_{k+2}=0$.
We denote the cohomology of the periodic cyclic complex $CC_{per}^\bullet(\ca)$ by $HC_{per}^\bullet(\ca)$; these again take two distinct values $HC_{per}^0(\ca)$ and $HC_{per}^1(\ca)$.

Note that it is important  that the cyclic cohomological bicomplex is totalized by using  direct sums and not products: if we remove the requirement that only finitely may of the $\phi_k$ are nonzero we obtain a complex with vanishing cohomology. One can however obtain a nontrivial theory with infinite cochains as follows. Assume that $\ca$ is a Banach (or normed) algebra. Denote by $CC_{entire}^i(\ca)$
the space  of cochains $\phi_k \in C^k(\ca) $, $k$ same parity as $i$, satisfying
\[
\sum \sqrt{k!}\|\phi_k\|r^k <\infty \text{ for every }  r\ge 0.
\]
The space of such cochains is preserved by the differential $b+B$ and we thus obtain a complex $CC_{entire}^\bullet(\ca)$ with the cohomology denoted by
$HC_{entire}^\bullet(\ca)$. $CC_{per}^\bullet(\ca)$ is clearly a subcomplex of $CC_{entire}^\bullet(\ca)$, and we therefore have a morphism
$HC_{per}^\bullet(\ca) \to HC_{entire}^\bullet(\ca)$ induced by the inclusion map.

\section{The longitudinal index formula  in cyclic theory}

%begin{example}
Let $(Z, F)$ be a foliated manifold, and $D$ a longitudinally elliptic operator on $Z$. Choose a (possibly disconnected) transversal $T$ to the foliation. By restricting the foliation groupoid to $T$, i.e. considering only paths which start and end in $T$ we obtain an etale groupoid $G_T$, Morita equivalent to the holonomy groupoid $G_{Z/F}$. One can consider the convolution algebra $C^\infty_c(G_T)$ of smooth compactly supported functions on $G_T$ and define the index of a leafwise family $D$ as an element
\[ \ind D \in K_0(C^\infty_c(G_T) \otimes \mathcal{R})\]
where $\mathcal{R}$ is the algebra of rapidly decaying infinite matrices.
We can apply a Chern character to obtain a class in cyclic homology $\Ch(\ind D) \in HC_0^{per}(C^\infty_c(G_T) \otimes \mathcal{R})\cong HC_0^{per}(C^\infty_c(G_T))$. To obtain  numerical information one has to pair this class with cyclic cohomology. The cyclic cohomology of  etale groupoid convolution algebras has been completely described in \cites{Brylinski-Nistor, Crainic}, based on earlier computations for group algebras and cross-products \cites{Burghelea,FT, Nistor, GJ}.
Earlier  Connes constructed a canonical imbedding $\Phi \colon H^\bullet(BG_T, \tau) \to HC_{per}^{\bullet-\dim T} ( C^\infty_c(G_T))$.
Here $BG_T$ is  the classifying space of the groupoid $G_T$ and $\tau$ is  the local system on $BG_T$ induced by the orientation bundle of $T$. We can now state Connes' index formula for a longitudinally elliptic family. Assume for simplicity that $D_V$ is a family of longitudinal Dirac operators with coefficients in a auxiliary bundle $V$.
Then we have \cite{Co94}
\[
\langle \Phi(c), \Ch(\ind D_V)  \rangle = \int_Z \widehat{A}(F) \Ch(V) \nu^*(c)
\]
Here $\nu \colon Z \to BG_T$ is the classifying map of $G_T$ and $\widehat{A}(F)$ is the $\widehat{A}$ genus of the bundle $F$.

%%%%%%%%%%%%%%%%%%%%%%%%%%
%{\color{red} What is $X$? Should it be $Z$?} --corrected
%%%%%%%%%%%%%%%%%%%%%%%%%%

We note that unlike the case of, for example, elliptic families, this result can't in general be deduced from the $K$-theoretical Connes-Skandalis index theorem. This is due to the fact that while $\ind D_V$ can be defined in $K_0(C^\infty_c(G_T) \otimes \mathcal{R})$, only its image in the $K$-theory of the $C^*$-completion of $C^\infty_c(G_T)$ is computed by the Connes-Skandalis index theorem. Cyclic homology and the Chern character on the other hand are nontrivial only for the smaller algebra $C^\infty_c(G_T) \otimes \mathcal{R}$.

The local index theory approach to the cohomological index formula for foliations and etale groupoids has been developed in \cite{gl1, gl2, g}.
%\end{example}

\section{Finitely summable Fredholm modules}
\begin{definition}
  An odd Fredholm module $(\mathcal{H},F)$ over an algebra $\ca$ is given by
the following data:

A Hilbert space $\mathcal{H}$
and a representation on it of an algebra $\ca$, i.e.
a homomorphism $\pi:\ca \to \End(\mathcal{H})$.

An  operator $F$, such that
\begin{align}
&\pi(a)(1-F^2)\in \ck \text{ for every }a \in \ca \\
&\pi(a)[F,\pi(b)] \in \ck \text{ for every }a, b \in \ca
\end{align}
\end{definition}
The set of (equivalence classes) of odd Fredholm modules with  the appropriate equivalence relation (cf. \cite{Bl98})
and an operation of direct sum becomes the $K$-homology group $K^1(\ca)$.
For a unital algebra $\ca$ (but not necessarily unital representation)
one can replace the Fredholm module by an equivalent one
satisfying
\begin{align}
&1-F^2\in \ck \text{ for every }a \in \ca\\
&[F,\pi(a)] \in \ck \text{ for every }a \in \ca
\end{align}

\begin{definition}
  An even Fredholm module $\left(\mathcal{H},F,\g\right)$ over an algebra $\ca$ is
given by the following data:

An odd Fredholm module $\left(\mathcal{H},F\right)$ over an algebra $\ca$

A $\mathbb{Z}_2$ grading on the Hilbert space $\mathcal{H}$, i.e. an operator $\gamma$ satisfying
$\gamma^2=1$. This data has to satisfy the following conditions:  the operator $F$ is odd with respect to this grading, i.e.
\[
F\gamma+\gamma F=0
\]
and elements of the algebra $\ca$ are even, i.e.
\[
\pi(a) \gamma= \gamma \pi(a) \text{ for every }a \in \ca.
\]

\end{definition}
As in the odd case, equivalence classes of even Fredholm modules form the group $K^1(\ca)$.
Let $p\ge 1$.
\begin{definition}
An (odd or even) or Fredholm module is $p$-summable if the following stronger conditions hold:
\begin{align}
&\pi(a)(1-F^2)\in \lp \text{ for every }a \in \ca \\
&\pi(a)[F,\pi(b)] \in \lp \text{ for every }a, b \in \ca
\end{align}
\end{definition}
where $\lp$ is the Schatten-ideal of operators $T$ with $\Tr(|T|^p)<\infty$.
In   \cite{cn85} Connes shows that with every Fredholm module
(called pre-Fredholm module in \cite{cn85}) one can canonically, by changing the Hilbert space and representation,
associate a  different Fredholm module, satisfying
\begin{equation}\label{f2}
F^2=1
\end{equation}
representing the same $K$-homology class. If the original Fredholm
module was $p$-summable, the new one will also be $p$-summable.

With the Fredholm module satisfying \eqref{f2} Connes associates a
cyclic cocycle called the character of the Fredholm module. Choose $n>p-1$ and the same parity as the
Fredholm module. Then, viewed in the periodic cyclic bicomplex, Connes' character has only one component
of degree $n$, $\tau_n$. defined by the equations

for an even Fredholm module
\begin{equation}
\tau_n(F)(a_0,a_1,\dots,a_n)=\frac{(n/2)!}{n!}\Tr'(\g \pi(a_0)[F,\pi(a_1)]\dots[F,\pi(a_n)])
\end{equation}

and for an odd Fredholm module
\begin{equation}
\tau_n(F)(a_0,a_1,\dots,a_n)=\ci \frac{\G(n/2+1)}{n!}\Tr'(\pi(a_0)[F,\pi(a_1)]\dots[F,\pi(a_n)])
\end{equation}
Here we use the notation
$\Tr'(T)=1/2\Tr\left(F(FT+TF)\right)$. The cohomology class of the cocycle $\tau_n$ in  periodic cyclic cohomology does not depend on the choice of $n$ (of appropriate parity).

A fundamental property of the character of a Fredholm module is the following theorem of Connes \cite{cn85}.
Let $e$ be an idempotent in $M_N(\ca)$, and $(\mathcal{H},F,\g)$ be an even Fredholm
module over $\ca$. Construct a Fredholm  operator
\[F_e = \pi(e)(F \otimes 1) \pi(e) \colon \mathcal{H}^+\otimes \mathbb{C}^N \rightarrow \mathcal{H}^- \otimes \mathbb{C}^N\]
 (where $\mathcal{H}^+$ and $\mathcal{H}^-$ denote the $\pm 1$ eigenspaces of $\gamma$).
Then
\begin{theorem} Let $(\mathcal{H}, F, \gamma)$ be an even $p$-summable Fredholm module satisfying $F^2=1$, $n> p-1$. Then
\begin{equation}\label{ondex}
\ind(F_e)=\langle \tau_n(F),\Ch(e)\rangle
\end{equation}
\end{theorem}
Here $\Ch(e)$ is the Chern character in  cyclic homology.

Similarly in the odd case pairing of the character of a Fredholm module with a class in odd $K$-theory computes the spectral flow. More precisely, let $(\mathcal{H},F)$ be an odd Fredholm module and $u$ an invertible element of $M_N(\ca)$.  Equivalently, if $P$ is the spectral projection on the positive spectrum of $F  $ (i.e. $P=1/2(1+F)$ if $F^2=1$) consider the Fredholm operator
\[T_u:=(P\otimes 1) \pi(u) (P\otimes 1) \in \End (P \mathcal{H} \otimes
\mathbb{C}^N).\]
 Then
\begin{theorem} Let $(\mathcal{H}, F)$ be an odd $p$-summable Fredholm module satisfying $F^2=1$, $n >p-1$. Then
\begin{equation*}\ind T_u=\langle \tau_n(F), \Ch(u) \rangle\end{equation*}
\end{theorem}
When $\pi$ is unital one can show that $\ind T_u$ coincides with the spectral flow  between two selfadjoint operators   $F\otimes 1$ and $ \pi(u)((F\otimes
1)\pi(u)^{-1}$ acting on the space
$\mathcal{H} \otimes
\mathbb{C}^N $.

\emph{In what follows we fix the representation of $\ca$ on $\mathcal{H}$ and will therefore drop it from the notation.
Moreover, we will assume that the algebra $\ca$ is unital and its representation is unital as well.}

Assume now that we have  an even or odd Fredholm module satisfying
\begin{align}
&(1-F^2)\in \mathcal{L}^{\frac{p}{2}} \label{p/2}\\
&[F,a] \in \lp \text{ for } a \in \ca \label{p-sum}.
\end{align}
We remark that for any $p$ summable Fredholm
module one can
achieve these summability conditions by perturbing the operator $F$ and
keeping all the other data intact. Then one can directly construct cyclic cocycles
on $\ca$ representing the class of Connes' character of a Fredholm module \cite{gor}.

In the even case: choose even $n>p$. Define a periodic cyclic cocycle with components $\Ch^k_n(F)$, $k=0, 2, \ldots n$ by
\begin{multline} \label {fredchern}
\Ch^k_{n}(F)(a_0,a_1,\dots a_k)=\\
\frac{(\frac{n}{2})!}{(\frac{n+k}{2})!}\sum_{i_0+i_1+
\dots + i_k=\frac{n-k}{2}}\Tr\,\g a_0 (1-F^2)^{i_0}[F, a_1](1-F^2)^{i_1}
\dots [F,a_k] (1-F^2)^{i_k}
\end{multline}

Similarly in the odd case choose odd $n>p$. Define a periodic cyclic cocycle with components $\Ch^k_n(F)$, $k=1, 3, \ldots n$ by
\begin{multline} \label{fredchern1}
\Ch^k_{n}(F)(a_0,a_1, \dots ,a_k)=
\frac{\G(\frac{n}{2}+1)\ci}{(\frac{n+k}{2})!}\\
\sum_{i_0+i_1+
\dots + i_k=\frac{n-k}{2}}\Tr\,  a_0(1-F^2)^{i_0} [F,a_1](1-F^2)^{i_1}
\dots [F,a_k](1-F^2)^{i_k}
\end{multline}

One can slightly improve the summability assumptions and require only $n>p-1$ (rather than $n>p$) by replacing $\Tr$ with $\Tr'$ where now
$\Tr'(T)=1/2\Tr\left(F(FT+TF)\right)+\Tr(1-F^2)T$.

These cocycles represent the class of the character of a Fredholm module in the following sense. With every even $p$-summable Fredholm module $(\mathcal{H}, F, \gamma)$ one can associate by Connes' construction a Fredholm module $(\mathcal{H}', F', \gamma')$ satisfying $(F')^2=1$.
\begin{theorem} Let $(\mathcal{H}, F, \gamma)$ be a $p$-summable Fredholm module satisfying condition \eqref{p/2}. Then
\[
\left[ \Ch_n(F)\right]=\left[ \tau_n(F')\right], \ n>p-1.
\]

\end{theorem}
An analogous result holds in the odd case.

\section{Unbounded picture}
\begin{definition}
A $p$-summable \emph{spectral triple} (or unbounded Fredholm module) $(\ca, \mathcal{H}, D)$ consists of a unital algebra $\ca$ represented on a Hilbert space $\mathcal{H}$ and a selfadjoint operator $D$ such that
\begin{equation}\label{bddcomm}
\text{For any $a \in \ca$, $a(\Dom D) \subset \Dom D$ and $[D, a]$ is bounded, and}
\end{equation}\label{psummable}
 \begin{equation}\label{psum}
          (D^2+1)^{-1} \in \cl^{p/2}
        \end{equation}
An \emph{even} $p$-summable spectral triple is a spectral triple   $(\ca, \mathcal{H}, D)$ with the following additional data:
$\mathcal{H}$ is equipped with a $\mathbb{Z}_2$ grading given by an operator $\gamma$ satisfying $\gamma^2=1$. We denote by $\mathcal{H}^\pm$ the $\pm 1$ eigenspaces of $\gamma$. The operator $D$ is odd with respect to this grading, i.e. $\gamma (\Dom D) \subset \Dom D$ and $D\gamma+ \gamma D=0$. We will assume for simplicity that $\ca$ is not graded and represented by even operators.

Otherwise a spectral triple is \emph{odd}.
\end{definition}

\begin{example} \label{Dirac}
Let $M$ be a compact $spin^c$ manifold and $E$ the associated spinor bundle. With it one can associate a spectral triple as follows: $\mathcal{H} =L^2(E)$ is the Hilbert space of $L^2$-sections of $E$. $\ca= C^\infty(M)$ acts on $\mathcal{H}$ by multiplication. Finally $D$ is the Dirac operator on $E$. This spectral triple will be $p$-summable for every $p > dim M$. If the dimension of $M$ is even, $E$ has a natural grading which anticommutes with $D$. Thus we obtain an even finitely summable spectral triple for an even-dimensional $M$ and odd for an odd-dimensional one.
\end{example}
If $(\ca, \mathcal{H}, D)$ is a $p$-summable spectral triple, one can form an associated Fredholm module $(\ca, \mathcal{H}, F)$ where
$
F = D(D^2+1)^{-1/2}.
$
It is easy to see that $(\ca, \mathcal{H}, F)$ is $p$-summable and, moreover, satisfies condition \eqref{p/2}.
One can therefore define its character using \eqref{fredchern}, \eqref{fredchern1}. An alternative approach to the character in  this case is given by the JLO formula \cite{JLO, GS}. It is applicable in a more general context of $\theta$-summable Fredholm modules, i.e. triples $(\ca, \mathcal{H}, D)$ satisfying equation \eqref{bddcomm} and with finite summability condition \eqref{psum} replaced by
\begin{equation}\label{thetasummable}
  e^{-\theta D^2} \in \cl^1 \text{ for any } \theta >0.
\end{equation}
Note that condition \eqref{psum} implies \eqref{thetasummable} since $ e^{-\theta D^2} = (1+D^2)^{-p/2} \left((1+D^2)^{p/2}  e^{-\theta D^2}\right)$ and the operator
$(1+D^2)^{p/2}  e^{-\theta D^2}$ is bounded (by the spectral theorem).

The JLO (Jaffe-Lesniewski-Osterwalder, \cite{JLO}) formula associates with every even $\theta$-summable Fredholm module an infinite  cochain in the $(b, B)$ bicomplex of $\ca$  given for $k=0, 2, \ldots$ by
\begin{multline}\label{JLO-even}
\Ch^k(D) (a_0, a_1, \ldots a_k) =\\
\int_{\Delta^k } \Tr \gamma a_0 e^{-t_0 D^2} [D, a_1] e^{-t_1D^2} \ldots [D, a_k] e^{-t_k D^2} dt_1dt_2\ldots dt_k
\end{multline}
where $t_i\ge 0$, $i=0, \ldots, k$, $\sum t_i=1$ are the barycentric coordinates on the simplex $\Delta^k$.

The formula in the odd case is similar:
\begin{multline}\label{JLO-odd}
\Ch^k(D) (a_0, a_1, \ldots a_k) =
\ci \int_{\Delta^k } \Tr  a_0 e^{-t_0 D^2} [D, a_1] e^{-t_1D^2} \ldots [D, a_k] e^{-t_k D^2} dt_1dt_2\ldots dt_k
\end{multline}
$k=1,3,\ldots$.

\begin{theorem}(\cite{JLO})
\begin{enumerate}
  \item The cochain $\Ch_\bullet(D)$ is an entire cyclic cocycle.
  \item The cocycles $\Ch_\bullet(D)$ and $\Ch_\bullet(\epsilon D)$ are canonically cohomologous.
\end{enumerate}
\end{theorem}
The second statement follows from the following transgression formula:
\[
-\frac{d}{d\epsilon }\Ch^k(D) =b \Ch^{k-1}(\epsilon D, D)+B\Ch^{k+1}(\epsilon D, D)
\]
where, in the even case  $\Ch(\epsilon D, D)$ is an odd cochain given by ($k$-odd)
\begin{multline}\label{transeven}
\Ch^k(\epsilon D, D) (a_0, a_1, \ldots a_k) =
\sum_{l=0}^{k}(-1)^l \int_{\Delta^{k+1}} \Tr \gamma a_0 e^{-t_0 D^2} [D, a_1]\\ e^{-t_1D^2} \ldots  e^{-t_lD^2}D e^{-t_{l+1}D^2} \ldots [D, a_k] e^{-t_{k+1} D^2} dt_1dt_2\ldots dt_{k+1}.
\end{multline}
In the odd case $\Ch(\epsilon D, D)$ is an even cochain given by ($k$ -even)
\begin{multline}\label{transodd}
\Ch^k(\epsilon D, D) (a_0, a_1, \ldots a_k) =
\ci \sum_{l=0}^{k}(-1)^l \int_{\Delta^{k+1}} \Tr   a_0 e^{-t_0 D^2} [D, a_1]\\ e^{-t_1D^2} \ldots  e^{-t_lD^2}D e^{-t_{l+1}D^2} \ldots [D, a_k] e^{-t_{k+1} D^2} dt_1dt_2\ldots dt_{k+1}.
\end{multline}
Finally, the following result of \cite{cmt2} shows that the JLO cocycle indeed computes Connes' character in the following sense:
\begin{theorem}(\cite{cmt2})
If $(\ca, \mathcal{H}, D)$ is $p$-summable and $F= D(1+D^2)^{-1/2}$ then the image   of $[\Ch_\bullet(F)]$  in $HC_{entire}^\bullet(\ca)$ is $[\Ch_\bullet(D)]$.
\end{theorem}

\section{Locality and spectral invariants}
It is well known that indices of elliptic operators, e.g. the Dirac operator from  Example \ref{Dirac}  can be computed as integrals of well defined local densities.
All of the formulas previously described above apply in the case of Example \ref{Dirac} and in conjunction with equation \eqref{ondex} provide formulas for the
index of the Dirac operator twisted with a vector bundle. A common feature of these formulas however is that they express the index in terms of traces of operators and thus are not local. One can however obtain a local expression from these formulas e.g. by replacing $D$ by $\epsilon D$ in the JLO formula, $\epsilon>0$ and considering the limit when $\epsilon \to 0^+$.
One therefore is naturally led \cite{Co94} to the question of obtaining a local formula for Connes' Chern character for spectral triples. Invariants of noncommutative geometry naturally have a spectral nature. A prototypical example of a local spectral invariant appearing in geometry is the noncommutative residue introduced in \cite{wodz}, cf. also \cite{guil}.
We recall the definition.

Let $X$ be a compact manifold. Denote by $\Psi(X)$ the algebra of classical pseudodifferential operators of integral orders, $\Psi^k(X)$ denotes the space of operator of order $k$.
Choose a positive pseudodifferential operator $R$ of order $1$. For  $A \in \Psi(X)$ consider the function
\begin{equation*}
\zeta_A(s):= \Tr A R^{-s}.
\end{equation*}
$\zeta(s)$ is defined for $\Re(s) >\dim X + \ord A$ and admits a meromorphic extension to the entire complex plane \cite{seely}.
It has at most a simple pole at $s=0$.

Introduce
\begin{equation*}
\Res A:= \Res_{s=0} \zeta(s)
\end{equation*}
It has the following properties
\begin{enumerate}
  \item $\Res A$ does not depend on the choice of $R \in \Psi^1(X)$,
  \item $\Res$ is a trace on the algebra $\Psi(X)$:
  \begin{equation*}\Res AB = \Res BA \text{ for } A, B \in \Psi(X),\end{equation*}
  \item $\Res A = \Res B$ whenever $A-B \in \Psi^{-\dim X -1}(X) $
\end{enumerate}
The last property is a manifestation of the locality of $\Res$: it depends only on a (part of) the complete symbol of an operator.

Explicitly in local coordinates noncommutative residue can be described as follows. Let $U\subset X$ be a coordinate neighborhood, and $(x, \xi)$ -- the standard coordinates on $T^*U$.
Let
\[
a(x, \xi)\sim \sum_{k=-\infty}^{m} a_k(x, \xi)
\]
be the asymptotic expansion of the complete symbol $a(x, \xi)$ of the operator $A\in \Psi^m(M)$, $a_k(x, \xi)$ homogeneous of order $k$ on in $\xi$.
Then the expression
\[
\left(\int_{|\xi|=1} a_{-n}(x, \xi) \vol_S \right) |dx|,
\]
where $\vol_S$ denotes the normalized volume form on the unit sphere $|\xi|=1$,
defines a density on $U$ independent of the choice of local coordinates.
\[
\Res(A) = \int_X \left(\int_{|\xi|=1} a_{-n}(x, \xi)\vol_S\right) |dx|.
\]

\section{Pseudodifferential calculus for spectral triples}
Motivated by connections between the noncommutative residue and the Dixmier trace, Connes and Moscovici extended it to more general spectral triples.
As a first step they construct an analogue of the pseudodifferential calculus for spectral triples.

Let $P_{\Ker(D^2)}$ be orthogonal projection onto $\Ker(D^2)$.
Define
\begin{equation} \label{1.2}
|D| := \sqrt{D^2} + P_{\Ker(D^2)}.
\end{equation}
If $D$ is invertible then $|D|$ has the usual meaning, but for us $|D|$
is always a strictly positive operator.

For  nonnegative $s$, put $\mathcal{H}^s = \Dom \left( |D|^s \right)$,
with the inner product
\begin{equation} \label{1.3}
\langle v_1, v_2 \rangle_{\mathcal{H}^s} =
\langle |D|^s v_1, |D|^s v_2 \rangle_{\mathcal{H}}.
\end{equation}
For $s <0$, put $\mathcal{H}^s = \left( \mathcal{H}^{-s} \right)^*$.
Put $\mathcal{H}^\infty = \bigcap_{s \ge 0} \mathcal{H}^s$, a dense subspace of $\mathcal{H}$.

Following
\cite[Appendix B]{CM95}, we can consider operators acting in a controlled way in this scale.

Let $op^k$ be the set of closed operators $F$ such that
\begin{enumerate}
\item  $\mathcal{H}^\infty \subset \Dom(F)$,
\item $F( \mathcal{H}^\infty) \subset \mathcal{H}^\infty$, and
\item For all $s$,  the operator $F \colon \mathcal{H}^\infty \rightarrow \mathcal{H}^\infty$ extends
to a bounded operator from $\mathcal{H}^s$ to $\mathcal{H}^{s-k}$.
\end{enumerate}

Introduce now a derivation $\delta$ on $\cb(\mathcal{H})$ by
\begin{equation*}
\delta(T) := [|D|, T]
\end{equation*}
with the domain $\Dom \delta$ consisting of  $T \in \cb(\mathcal{H})$ such that  $T(\Dom D) \subset \Dom D$ and $[|D|, T]$ extends to a bounded operator.
Define $OP^0 : = \cap \left\{ \Dom \delta^n \mid n \in \mathbb{N}\right\}$.
It is shown in \cite[Appendix B]{CM95} that
\begin{equation*}
OP^0 \subset op^0.
\end{equation*}
Define now $OP^\alpha$, $\alpha \in \mathbb{R}$ as the set of closed operators $P$ for which $|D|^{-\alpha}P \in OP^0$. Note that $OP^\alpha \subset op^\alpha$.
Connes and Moscovici show that with
\begin{equation*}
\nabla(T):=[D^2, T]
\end{equation*}
$\nabla (OP^\alpha )\subset OP^{\alpha+1}$. On the other hand $\delta(OP^\alpha) \subset OP^\alpha$.

 Carefully estimating the remainder one can prove that for $T\in OP^\alpha$ there is an asymptotic expansion
\begin{equation}\label{asymp-1}
|D|^{2z} \cdot T \cdot |D|^{-2z} \simeq \sum_{k=0}^{\infty} \frac{z(z-1)\ldots (z-k)}{k!} \nabla^k(T)|D|^{-2k}
\end{equation}
The precise meaning of the asymptotic expansion is that  for every $N \in \mathbb{N}$
\begin{equation*}
|D|^{2z} \cdot T \cdot |D|^{-2z} - \sum_{k=0}^{N} \frac{z(z-1)\ldots (z-k+1)}{k!} \nabla^k(T)|D|^{-2k} \in OP^{ \alpha -N-1}
\end{equation*}
(note that $\nabla^k(T)|D|^{-2k} \in OP^{\alpha-k}$).
For an invertible $D$ we have $\nabla(T)|D|^{-2} = 2 \delta(T)|D|^{-1}+ \delta^2(T)|D|^{-2}$; for not necessarily invertible $D$  $\nabla(T)|D|^{-2} - 2 \delta(T)|D|^{-1}- \delta^2(T)|D|^{-2}$ is a finite rank operator and hence is in $OP^{-\infty}$.
 Using this we can rewrite the asymptotic expansion as
\begin{equation}\label{asymp}
|D|^{z} \cdot T \cdot |D|^{-z} \simeq \sum_{k=0}^{\infty} \frac{z(z-1)\ldots (z-k)}{k!} \delta^k(T)|D|^{-k}
\end{equation}

It will also be convenient to consider a derivation
\begin{equation*}
L(T):= [ \log |D|^2, T]= \frac{d}{dz}|_{z=0}|D|^{2z} \cdot T \cdot |D|^{-2z}
\end{equation*}
It follows from \eqref{asymp-1}, \eqref{asymp} that for $T \in OP^{\alpha}$  $L(T) \in OP^{\alpha-1}$ and there are asymptotic expansions (in the above sense)
\begin{equation*}
L(T) \simeq \sum_{k=1}^{\infty} \frac{(-1)^{k-1}}{k} \nabla^k(T)|D|^{-2k},\  L(T) \simeq 2 \sum_{k=1}^{\infty} \frac{(-1)^{k-1}}{k} \delta^k(T)|D|^{-k}.
\end{equation*}
Finally, for $T \in OP^\alpha$ we have the following equality
\begin{equation*}
|D|^{2z} \cdot T \cdot |D|^{-2z} = \sum_{k=0}^{\infty} \frac{z^k}{k!} L^k(T)
\end{equation*}
where convergence of the series on the right is in each of the norms $T \mapsto \| \delta^n(|D|^{-\alpha}T) \|$, $n=0,1,2\ldots$.

Impose from now on the following smoothness assumption on the spectral triple: for every $a \in \ca$ we have $a \in OP^0$, $[D, a] \in OP^0$. Let $\cb$ be the algebra generated by $\delta^n(a)$, $\delta^n([D,a])$, $a\in \ca$, $n \ge 0$. Clearly $\cb \subset OP^0$. One defines the  pseudodifferential operators of order $\alpha$ by
\begin{equation*}
\Psi^\alpha(\ca) =\left\{ P\in   OP^{\Re \alpha} \mid P\simeq \sum_{k=0}^{\infty}b_k|D|^{\alpha-k}, \ b_k\in \cb  \right\}.
\end{equation*}
It follows from the asymptotic expansion \eqref{asymp} that $\Psi^\alpha \cdot \Psi^\beta \subset \Psi^{\alpha+\beta}$. In particular we can consider the algebra
\begin{equation*}
\Psi(\ca):= \bigcup_{k \in \mathbb{Z}} \Psi^k(\ca).
\end{equation*}
\section{Dimension Spectrum}
From now on we assume that the spectral triple $(\ca, \mathcal{H}, D)$ is $p$-summable. Then $OP^\alpha \subset \cl^1(\mathcal{H})$ f for $\alpha \le -p$. It follows that for every $b \in \cb$ the function
\begin{equation*}
\zeta_b(s):= \Tr b |D|^{-s}
\end{equation*}
is defined and holomorphic for $\Re s >p$.
\begin{definition} A spectral triple $(\ca, \mathcal{H}, D)$ has a discrete dimensional spectrum $Sd \subset \mathbb{C}$ if $Sd$ is a discrete set and $\zeta_b(s)$ extends holomorphically to $\mathbb{C}\setminus Sd$ for every $b \in \cb$.
\end{definition}
We will assume at the beginning that each $\zeta_b(s)$ actually extends meromophically to $\mathbb{C}$ with poles only in $Sd$ and the orders of every pole is at most $q \in \mathbb{N}$, $q$  independent of the pole and $b$.
It is immediate from the definition that for every $P \in \Psi^\alpha(\ca)$ the function $\zeta_P(s):= \Tr P|D|^{-s}$ is holomorphic for $\Re s $ sufficiently large and extends meromophically  to $\mathbb{C}$ with possible poles in the set
$\bigcup \left\{ Sd+\alpha -k \right\}$, where $k$ runs through nonnegative integers.

We now define a collection of linear functionals $\tau_{i}$ on $\Psi^\alpha(\ca)$ for every $\alpha$, $ -1\le i \le q-1$ by the following equality:
\begin{equation*}
\zeta_P(2z) = \tau_{q-1}(P)z^{-q}+\tau_{q-2}(P)z^{-q+1}+\ldots \tau_0(P)z^{-1} + \tau_{-1}(P)+O(|z|) \text{ near } z=0
\end{equation*}

Note that for $\Re z$ sufficiently large
\begin{equation*}
\Tr (P_1 P_2 |D|^{-2z})=\Tr ( P_2 \left(|D|^{-2z}P_1|D|^{2z} \right)|D|^{-2z})=
\sum_{k=0}^{\infty} \frac{(-z)^k}{k!}\Tr (P_2  L^k(P_1)|D|^{-2z})
\end{equation*}
By meromorphic continuation this equality holds for every $z \in \mathbb{C}$, except for a discrete set. Comparing the Laurent series at $z=0$ of the right and left hand sides we conclude that
\begin{equation*}
\tau_i(P_1P_2) = \tau_i(P_2P_1)+ \sum_{k=1}^{q-1-k} \frac{(-1)^k}{k!} \tau_{i+k}
\end{equation*}
In particular, $\tau_{q-1}$ is a trace on $\bigcup_{\alpha} \Psi^\alpha(\ca)$. It is important to note that $\tau_i$, $i \ge 0$ are  \emph{local} in the following sense:
\begin{equation*}
\tau_i(P) =0 \text{ if } P \in \cl^1(\mathcal{H}), i \ge 0.
\end{equation*}
The family $\tau_i$ thus generalizes the noncommutative residue to the spectral triple framework.
$\tau_{-1}$ on the other hand does not have this property:
\begin{equation*}
\tau_{-1}(P) =\Tr P \text{ if } P \in \cl^1(\mathcal{H}).
\end{equation*}
If $\alpha \notin \left\{ k-Sd \right\}$, $k \ge 0$ then $\tau_{-1}$ defines a trace on $\Psi^\alpha$ -- a generalization of the Kontsevich-Vishik trace.

\section{The local index formula in the odd case}
We are now ready to sketch the derivation of the Connes-Moscovici local index formula in noncommutative geometry. We start with the odd case, as it is in this case that
one can obtain a formula that is fully local. As in the geometric situation of the spectral triple defined by a Dirac operator, we obtain the formula by studying the behavior of the JLO formula under rescaling $D\to \epsilon D$ as $\epsilon \to 0^+$.

The starting point is the following result of \cite{cmt2}. First we recall the notion of the finite part of a function. Assume that a function $g(\epsilon)$, $\epsilon \in (0, T]$, can be written as
\begin{equation}\label{pf}
g(\epsilon) =  \sum_{i,j \ge 0} \alpha_{i, j} (\log \epsilon)^j\epsilon^{-\lambda_i}+ \sum_{k\ge 1} \beta_k(\log \epsilon)^k   +\psi(\epsilon)
\end{equation}
where the sum is finite, $\Re \lambda_i \le 0$, $\lambda_i\ne 0$ and $\psi \in C[0, T]$. Then the finite part of $g$ at $0$ is defined by
\begin{equation*}
\PF(g):=\psi(0).
\end{equation*}
\begin{theorem}
Let $(\ca, \mathcal{H}, D)$ be a p-summable spectral triple.
Assume  that the components of $\Ch(\epsilon D)$ and $\Ch(\epsilon D, D)$ (see equations \eqref{JLO-odd}, \eqref{transodd})  have  asymptotic behavior as in \eqref{pf}. Define a cochain $\psi_k$ by
\[
\psi_k(a_0, a_1, \ldots, a_k): = \PF(\Ch(\epsilon D)(a_0, a_1, \ldots, a_k))
\]
(here $k$ takes odd values for odd spectral triples and even values for even spectral triples). Then  $\psi_k=0$ for $k>p$ and
$\psi_k$ is a  periodic cyclic cocycle  whose cohomology class
coincides with $[\Ch(F)]$, $F= D(D^2+1)^{-1/2}$.

\end{theorem}
Therefore we study the asymptotic behavior of the expression
\begin{equation*}
\Ch^k(\epsilon D)=\int_{\Delta^k } \Tr   a_0 e^{-t_0 \epsilon^2 D^2} [\epsilon D, a_1] e^{-t_1\epsilon^2 D^2} \ldots [\epsilon D, a_k] e^{-t_k \epsilon^2 D^2} dt_1dt_2\ldots dt_k.
\end{equation*}
We start with the following identity (here $P \in OP^\alpha$):
\begin{equation*}
e^{-s D^2} P   = \sum_{m=0}^{N}\frac{(-s)^m}{m!} \nabla^m(P) e^{-s D^2} + \frac{(-s)^{N+1}}{N!}\int_{0}^{1}(1-t)^N e^{-tsD^2} \nabla^{N+1}(P) e^{-(1-t)sD^2} dt.
\end{equation*}
Applying it repeatedly  to   $\Ch(\epsilon D)$ to move all the exponentials in
\[a_0 e^{-t_0 \epsilon^2 D^2} [  D, a_1] e^{-t_1\epsilon^2 D^2} \ldots [  D, a_k] e^{-t_k \epsilon^2 D^2} \]
 to the end of the expression we obtain
\begin{equation*}
\Ch^k(\epsilon D)=  \ci \sum \limits_{0 \le m_i \le N}C_{\mathbf{m}} \epsilon^{k+2m}  \Tr a_0 \nabla^{m_1}([D, a_1]) \ldots \nabla^{m_k}([D, a_k]) e^{-\epsilon^2 D^2}+R_N
\end{equation*}
where $m= \sum_{i=1}^{k} m_i$ and
\begin{multline*}
C_{\mathbf{m}} = (-1)^m \int_{\Delta^k } \frac{t_0^{m_1} (t_0+t_1)^{m_2}\ldots (t_0+t_1+\ldots t_{k-1})^{m_{k}}}{m_1!m_2! \ldots m_{k}!} dt_1 dt_2 \ldots dt_k=\\
(-1)^m \frac{1}{(m_1+1)(m_1+m_2+2) \ldots (m_1+m_2+\ldots m_k+k)m_1!m_2! \ldots m_{k}!}.
\end{multline*}
The remainder $R_N$ can be bounded by a constant times $\epsilon^{(k+1)N-p}$ (cf. \cite{gl}) and thus does not contribute to $\PF(\Ch(\epsilon D))$ if sufficiently large $N$ is chosen. Hence it is sufficient to determine  the finite part of
\begin{equation*}
\epsilon^{k+2m}  \Tr  A e^{-\epsilon^2 D^2}, \ A= a_0\nabla^{m_1}([D, a_1]) \ldots \nabla^{m_k}([D, a_k]) \in OP^m
\end{equation*}
Note that
\begin{equation*}
e^{-\epsilon^2 D^2} = e^{-\epsilon^2 |D|^2}+ (1-e^{-\epsilon^2}) P_{\Ker(D^2)}.
\end{equation*}
Hence
\begin{equation*} \Tr  A e^{-\epsilon^2 D^2} = \Tr  A e^{-\epsilon^2 |D|^2} +O(\epsilon^2)
\end{equation*}
and the finite parts of $\epsilon^{k+2m}  \Tr  A e^{-\epsilon^2 D^2}$ and $\epsilon^{k+2m}  \Tr  A e^{-\epsilon^2 |D|^2}$ coincide for every $k$, $m$.
The expression $\Tr  A e^{-\epsilon^2 D^2}$ is related to the $\zeta$-function $\zeta_A(s) = \Tr A|D|^{-s}$ by the Mellin transform: the equation
$\Gamma(s) |D|^{-2s} = \int_{0}^{\infty} e^{-t|D|^2} t^{s-1} dt$ implies that
\begin{equation*}
\Gamma(s)\zeta_A(2s) = \Gamma(s) \Tr A|D|^{-2s} = \int_{0}^{\infty}\Tr A e^{-t|D|^2} t^{s-1} dt
\end{equation*}
To deduce the asymptotic expansion of $\Tr A e^{-t|D|^2}$ at $t=0$ from the information on the poles of $\Gamma(s)\zeta_A(2s)$ we need a technical assumption on the decay of
$\Gamma(s)\zeta_A(2s)$ on vertical lines in the complex plane. Under this assumption we obtain that
\begin{equation*}
\epsilon^{k+2m}  \Tr  A e^{-\epsilon^2 |D|^2} = \sum_j \sum_{k=0}^{q_j-1} \alpha_{j, k} \epsilon^{-2a_j} \log^{k} \epsilon +o(1)\text{ as } \epsilon \to 0^+
\end{equation*}
where $a_j$ are the poles of $\Gamma(s)\zeta_A(2s)$ in the halfplane $\Re s \ge k/2+m$ and $q_j$ -- the order of the pole at $a_j$. Moreover the constant term in this expansion is equal to $\Res_{s=k/2+m} \Gamma(s)\zeta_A(2s)$.  A virtually identical argument shows that $\Ch(\epsilon D, D)$ has the desired asymptotic behavior as $\epsilon \to 0^+$.

Summarizing, we obtain the following result of Connes and Moscovici:
\begin{theorem}[Connes-Moscovici] The formulas
\begin{multline}
\psi_k(a_0, a_1, \ldots a_k):=\\ \ci \sum_{m_i \ge 0}C_{\mathbf{m}} \Res_{s=k/2+m} \Gamma(s) \Tr a_0\nabla^{m_1}([D, a_1]) \ldots \nabla^{m_k}([D, a_k])|D|^{-2s}=\\
\ci \sum_{m_i \ge 0}C_{\mathbf{m}} \Res_{s=0} \Gamma(s+k/2+m) \Tr a_0\nabla^{m_1}([D, a_1]) \ldots \nabla^{m_k}([D, a_k])|D|^{-2s-k-2m}
\end{multline}
$k=1, 3, \ldots$ define an odd periodic cyclic cocycle  cohomologous to $\Ch(F)$, $F=D(D^2+1)^{-1/2}$.
\end{theorem}

One can rewrite this result to obtain the formula in terms of generalized residues.  Since for $h(s)$ holomorphic at $s=0$
\begin{equation*}
\Res_{s=0}h(s) \zeta_A(2s) = \sum_{l\ge 0} \frac{h^{(l)}(0)}{l!}  \tau_l(A)
\end{equation*}
We therefore obtain the following version of the previous theorem.

\begin{theorem}[Connes-Moscovici] The formulas
\begin{multline}
\psi_k(a_0, a_1, \ldots a_k)=\\ \ci \sum_{m_i \ge 0, l\ge 0 }  \frac{\Gamma^{(l)}(k/2+m)}{l!}C_{\mathbf{m}}  \tau_l \left( a_0\nabla^{m_1}([D, a_1]) \ldots \nabla^{m_k}([D, a_k])|D|^{-k-2m} \right)
\end{multline}
$k=1, 3, \ldots$ define an odd periodic cyclic cocycle  cohomologous to $\Ch(F)$, $F=D(D^2+1)^{-1/2}$.
\end{theorem}

\section{Renormalization}
Here we outline the Connes-Moscovici process of renormalization which allows one to replace the local index cocycle $\psi_k$ by a cohomologous one of the form
\begin{multline*}
\psi_k'(a_0, a_1, \ldots a_k):=\\ \sum_{m_i \ge 0, l\ge 0 }C'(l, m_1, m_2, \ldots m_k)  \tau_l \left( a_0\nabla^{m_1}([D, a_1]) \ldots \nabla^{m_k}([D, a_k])|D|^{-k-2m} \right)
\end{multline*}

where the constants $C'(l, m_1, m_2, \ldots m_k)$ have the following properties: they are rational multiples of $\sqrt{2  \pi i}$ and the summation over $l$ is finite even if the $\zeta$-functions have  isolated essential singularities and/or there is no bound on the order of poles that can occur in the $\zeta$-functions.
The starting point is the observation that the cohomology class of $\psi_k$ does not change under rescaling $D \to e^{-\mu/2} D$, $\mu \in \mathbb{R}$.
The cocycle $\psi_k$ is then replaced by a cohomologous cocycle
\begin{multline}
\psi_k^\mu(a_0, a_1, \ldots a_k):=\\
\ci \sum_{m_i \ge 0}C_{\mathbf{m}} \Res_{s=0}e^{\mu s} \Gamma\left(s+\frac{k}{2}+m\right) \Tr a_0\nabla^{m_1}([D, a_1]) \ldots \nabla^{m_k}([D, a_k])|D|^{-2s-k-2m}=\\
\psi_k(a_0, a_1, \ldots a_k)+ \sum_{\mu \ge 1} \frac{\mu^q}{q!}\phi_k^q(a_0, a_1, \ldots a_k).
\end{multline}
where
\begin{multline*}
\phi_k^q(a_0, a_1, \ldots a_k):=\\
\ci \sum_{m_i \ge 0}C_{\mathbf{m}} \Res_{s=0} s^q\Gamma\left(s+\frac{k}{2}+m\right) \Tr a_0\nabla^{m_1}([D, a_1]) \ldots \nabla^{m_k}([D, a_k])|D|^{-2s-k-2m}
\end{multline*}
Here we used the identity
\begin{equation*}
\left| e^{-\mu/2}D\right|^{-2s} = e^{\mu s}|D|^{-2s} + (1-e^{\mu s}) P_{\ker D^2}
\end{equation*}
to replace $\left| e^{-\mu/2}D\right|^{-2s}$ by $e^{\mu s}|D|^{-2s}$ without changing the residues.
Hence for each $q=1, 2, \ldots $, $\phi_k^q$ is an odd cyclic cocycle cohomologous to $0$.
It follows that for every function $g(s)$ holomorphic at $0$ with $g(0)=1$ the formula
\begin{multline}\label{renorm}
\psi_k'(a_0, a_1, \ldots a_k):=\\ \ci \sum_{m_i \ge 0}C_{\mathbf{m}} \Res_{s=0} g(s) \Gamma\left(s+\frac{k}{2}+m\right)\Tr a_0\nabla^{m_1}([D, a_1]) \ldots \nabla^{m_k}([D, a_k])|D|^{-2s-k-2m}
\end{multline}
defines a periodic cyclic cocycle cohomologous to $\psi_k$. In  Connes-Moscovici renormalization one chooses $g(s) = \frac{\Gamma(1/2)}{\Gamma(1/2+s)}$.
One can then express the resulting formula in terms of generalized residues using the identity
\begin{equation*}
\frac{\Gamma\left(\frac{1}{2}\right)}{\Gamma\left(\frac{1}{2}+s\right)}\Gamma\left(s+\frac{k}{2}+m\right)= \Gamma\left(\frac{1}{2}\right)  \left(\frac{1}{2}+s\right)\left(\frac{3}{2}+s\right)\ldots \left(\frac{k+2m-2}{2}+s\right)
\end{equation*}
Denote by $\sigma_j(n)$ the coefficients in the expansion
\begin{equation*}
\left(\frac{1}{2}+s\right)\left(\frac{3}{2}+s\right)\ldots \left(\frac{2n-1}{2}+s\right) =\sum_{l=0}^{\infty} \sigma_{l}(n)s^{l}
\end{equation*}

\begin{theorem}[Connes-Moscovici] The formulas
\begin{multline*}
\psi_k'(a_0, a_1, \ldots a_k):=\\ \ci \sum_{m_i \ge 0, l\ge 0 }\sigma_l\left(\frac{k-1}{2}+m\right)C_{ \mathbf{m}}  \tau_l \left( a_0\nabla^{m_1}([D, a_1]) \ldots \nabla^{m_k}([D, a_k])|D|^{-k-2m} \right)
\end{multline*}

$k=1, 3, \ldots$ define an odd periodic cyclic cocycle  cohomologous to $\Ch(F)$, $F=D(D^2+1)^{-1/2}$,
\end{theorem}
Note that $\sigma_l\left(\frac{k-1}{2}+m\right) =0$ for $l >\frac{k-1}{2}+m$  and
\begin{equation*}a_0\nabla^{m_1}([D, a_1]) \ldots \nabla^{m_k}([D, a_k])|D|^{-k-2m} \in OP^{-k-m} \subset \cl^1
\end{equation*} when $k+m\ge p$ (recall that $p$ denotes the summability degree of the spectral triple). Since each $\tau_l$ vanishes on trace class operators, only $\tau_l$ with $l <p$ may appear in the renormalized formula.

\begin{example}
Let $(\ca, \mathcal{H}, D)$ be a spectral triple with discrete dimension spectrum. Let $\widetilde{\ca}$ be the algebra generated by $\ca$ and operators of the form $|D|^{-k}(\log |D|)^l$, where $k$, $l$ are integers, $k \ge 1$, $l \ge 0$. Since $\Tr \left(A\log|D|\right)|D|^{-s}=-\frac{d}{ds} \Tr A|D|^{-s}$, $(\widetilde{\ca}, \mathcal{H}, D)$ again will be a spectral triple with discrete dimension spectrum. Moreover, if  (at least for some $A\in \Psi(\ca)$)  $\zeta$-function $\Tr A|D|^{-s}$  is not entire,
the $\zeta$-functions with $A \in \Psi(\widetilde{\ca})$ will have poles of arbitrarily high orders.

In the situation of Example \ref{Dirac} the corresponding algebra of  pseudodifferential operators is contained in the algebra of pseudodifferential operators with
log–polyhomogeneous symbols constructed and studied in details in \cite{lesch}.

\end{example}

\section{The even case}
Most of the discussion above extends to the even case verbatim, so here we just state the relevant results and indicate the point where a difference with the odd case arises.

One first obtains the following result: the formulas
\begin{multline}
\psi_k(a_0, a_1, \ldots a_k):= \\
\sum_{m_i \ge 0}C_{\mathbf{m}} \Res_{s=0} \Gamma(s+k/2+m) \Tr \gamma a_0\nabla^{m_1}([D, a_1]) \ldots \nabla^{m_k}([D, a_k])|D|^{-2s-k-2m}
\end{multline}
$k=0, 2, \ldots$ define an even periodic cyclic cocycle  cohomologous to $\Ch(F)$, $F=D(D^2+1)^{-1/2}$.
Here, as in the odd case,
\begin{equation*}
C_{\mathbf{m}} = \\
 \frac{(-1)^m}{(m_1+1)(m_1+m_2+2) \ldots (m_1+m_2+\ldots m_k+k)m_1!m_2! \ldots m_{k}!},
\end{equation*}
$m=m_1+m_2+\ldots+m_k$

The renormalization process allows us to replace the cocycle $\psi_k$ by a cohomologous one
\begin{multline}
\psi_k'(a_0, a_1, \ldots a_k):= \\
\sum_{m_i \ge 0}C_{\mathbf{m}} \Res_{s=0} \frac{\Gamma\left(s+\frac{k}{2}+m\right)}{\Gamma(s+1)} \Tr \gamma a_0\nabla^{m_1}([D, a_1]) \ldots \nabla^{m_k}([D, a_k])|D|^{-2s-k-2m}
\end{multline}
Denote by $\sigma_j(n)$ coefficients in the expansion
\begin{equation*}
\frac{\Gamma(s+n)}{\Gamma(s+1)}=\left(1+s\right)\left(2+s\right)\ldots \left(n-1+s\right) =\sum_{l=0}^{\infty} \sigma_{l}(n)s^{l}
\end{equation*}
Then we can write an expression for the cocycle $\psi_k'$ in terms of linear functionals $\tau_l$:
\begin{theorem}[Connes-Moscovici] The formulas
\begin{align*}
&\psi_0'(a_0):= \tau_{-1}(\gamma a_0)\\
&\psi_k'(a_0, a_1, \ldots a_k):=\\
&\sum_{m_i \ge 0, l\ge 0 }\sigma_l\left(\frac{k}{2}+m\right)C_{\mathbf{m}}  \tau_l \left( a_0\nabla^{m_1}([D, a_1]) \ldots \nabla^{m_k}([D, a_k])|D|^{-k-2m} \right) \text{ for } k=2, 4, \ldots
\end{align*}

 define an even periodic cyclic cocycle  cohomologous to $\Ch(F)$, $F=D(D^2+1)^{-1/2}$.
\end{theorem}
Note the important difference with the odd case: in addition to generalized residues this formula contains a nonlocal term involving $\tau_{-1}$. Some nonlocality is unavoidable in the even case as can be seen by considering the case when $\mathcal{H}$ is finite dimensional. All the linear functionals $\tau_i$, $i \ge 0$ vanish, as every operator is of trace class. Nevertheless one can have operators with nonzero index between finite dimensional vector spaces and thus the Connes character of a finite dimensional spectral triple can  be nonzero.

\bibliographystyle{plain}
\bibliography{MyBibfile}

\end{document}